\documentclass[a4paper, 10pt]{amsart}
\usepackage{amsmath}
\usepackage{amssymb, color, hyperref}
\usepackage{amsfonts}

\theoremstyle{plain}
\newtheorem{theorem}{\bf Theorem}[section]
\newtheorem{proposition}[theorem]{\bf Proposition}
\newtheorem{lemma}[theorem]{\bf Lemma}

\theoremstyle{definition}

\newcommand{\N}{\mathbb N}
\newcommand{\Z}{\mathbb Z}

\renewcommand{\t}{\, | \,}
\newcommand{\und}{\;\mbox{ and }\;}

\newcommand{\la}{\langle}
\newcommand{\ra}{\rangle}

 \DeclareMathOperator{\ord}{ord}

 \DeclareMathOperator{\supp}{supp}

\newcommand{\bdot}{\boldsymbol{\cdot}}

\newcommand{\blfloor}{\big\lfloor}
\newcommand{\brfloor}{\big\rfloor}

\numberwithin{equation}{section}

\begin{document}

\title[minimal product-one sequences of maximal length]{On minimal product-one sequences of maximal length \\ over Dihedral and Dicyclic groups}

\author{Jun Seok Oh and Qinghai Zhong}
\address{Institute for Mathematics and Scientific Computing \\ University of Graz, NAWI Graz \\ Heinrichstra{\ss}e 36 \\ 8010 Graz, Austria}
\email{junseok.oh@uni-graz.at, qinghai.zhong@uni-graz.at}
\urladdr{https://imsc.uni-graz.at/zhong/}


\thanks{This work was supported by the Austrian Science Fund FWF, W1230 Doctoral Program Discrete Mathematics and Project No. P28864--N35.}

\begin{abstract}
Let $G$ be a finite group. By a sequence over $G$, we mean a finite unordered sequence of terms from $G$, where repetition is allowed, and we say that it is a product-one sequence if its terms can be ordered such that their product equals the identity element of $G$. The large Davenport constant $\mathsf D (G)$ is the maximal length of a minimal product-one sequence, that is, a product-one sequence which cannot be factored into two non-trivial product-one subsequences. We provide explicit characterizations of all minimal product-one sequences of length $\mathsf D (G)$ over Dihedral and Dicyclic groups. Based on these characterizations we study the unions of sets of lengths of the monoid of product-one sequences over these groups.
\end{abstract}

\maketitle


\bigskip
\section{Introduction} \label{1}
\bigskip

Let $G$ be a finite group. A sequence $S$ over $G$ means a finite sequence of terms from $G$ which is unordered, repetition of terms allowed. We say that $S$ is a product-one sequence if its terms can be ordered so that their product equals the identity element of the group. The {\it small Davenport constant $\mathsf d (G)$} is the maximal integer $\ell$ such that there is a sequence of length $\ell$ which has no non-trivial product-one subsequence. The {\it large Davenport constant $\mathsf D (G)$} is the maximal length of a minimal product-one sequence (this is a product-one sequence which cannot be factored into two non-trivial product-one subsequences). We have $1 + \mathsf d (G) \le \mathsf D (G)$ and equality holds if $G$ is abelian. The study of the Davenport constant of finite abelian groups has been a central topic in zero-sum theory since the 1960s (see \cite{Ga-Ge06b} for a survey). Both the direct problem, asking for the precise value of the Davenport constant in terms of the group invariants, as well as the associated inverse problem, asking for the structure of extremal sequences, have received wide attention in the literature. We refer to \cite{Ga-Ge-Sc07a,Sc10b,Ga-Ge-Gr10a, Sc11b, Sa-Ch14a, Gi18a, Gi-Sc19a, Gi-Sc19b} for progress with respect to the direct and to the inverse problem. Much of this research was stimulated by and applied to factorization theory and we refer to \cite{Ge-HK06a,Ge09a} for more information on this interplay.

Applications to invariant theory (in particular, the relationship of the small and large Davenport constants with the Noether number, see \cite{Cz-Do14a, Cz-Do13c, Cz-Do-Ge16a, Cz-Do-Sz18, Cz19a, Ha-Zh19a}) pushed forward the study of the Davenport constants for finite non-abelian groups. Geroldinger and Grynkiewicz (\cite{Ge-Gr13a, Gr13b}) studied the small and the large Davenport constant of non-abelian groups and derived their precise values for groups having a cyclic index $2$ subgroup. Brochero Mart\'\i nez and Ribas (\cite{Br-Ri18a,Br-Ri19a}) determined, among others, the structure of product-one free sequences of length $\mathsf d (G)$ over Dihedral and Dicyclic groups.

In this paper we establish a characterization of the structure of minimal product-one sequences of length $\mathsf D (G)$ over Dihedral and Dicyclic groups (Theorems \ref{4.1}, \ref{4.2}, and \ref{4.3}). It turns out that this  problem is quite different from the study of product-one free sequence done by Brochero Mart\'\i nez and Ribas. The minimal product-one sequences over $G$ are the atoms (irreducible elements) of the monoid $\mathcal B (G)$ of all product-one sequences over $G$. Algebraic and arithmetic properties of $\mathcal B (G)$ were recently studied in \cite{Oh18a, Oh19a}. Based on our characterization results of minimal product-one sequences of length $\mathsf D (G)$ we give a description of all unions of sets of lengths of $\mathcal B (G)$ (Theorems \ref{5.4} and \ref{5.5}).

We proceed as follows. In Section \ref{2}, we fix our notation and gather the required tools. In Section \ref{3}, we study  the structure of minimal product-one sequences fulfilling certain requirements on their length and their support (Propositions \ref{3.2} and \ref{3.3}). Based on these preparatory results, we establish an explicit characterization  of all minimal product-one sequences having length $\mathsf D (G)$ for Dihedral groups (Theorems \ref{4.1} and \ref{4.2}) and for Dicyclic groups (Theorem \ref{4.3}). Our results on unions of sets of lengths are given in Section \ref{5}.


\bigskip
\section{Preliminaries} \label{2}
\bigskip

We denote by $\N$ the set of positive integers and we set $\N_0 = \N \cup \{ 0 \}$. For each $k \in \N$, we also denote by $\N_{\ge k}$ the set of positive integers greater than or equal to $k$. For integers $a, b \in \Z$, $[a,b] = \{x \in \Z \mid a \le x \le b\}$ is the discrete interval.

\smallskip
\noindent
{\bf Groups.} Let $G$ be a multiplicatively written finite group with identity element $1_G$. For an element $g \in G$, we denote by $\ord (g) \in \N$ the order of $g$, and for subsets $A, B \subset G$, we set
\[
  AB \, = \, \{ ab \t a \in A \und b \in B \} \quad \und \quad gA \, = \, \{ ga \t a \in A \} \,.
\]
If $G_0 \subset G$ is a non-empty subset, then we denote by $\la G_0 \ra \subset G$ the subgroup generated by $G_0$, and by $\mathsf H (G_0) = \{ g \in G \t gG_0 = G_0 \}$ the left {\it stabilizer} of $G_0$. Then $\mathsf H (G_0) \subset G$ is a subgroup, and $G_0$ is a union of right $\mathsf H (G_0)$-cosets. Of course, if $G$ is abelian, then we do not need to differentiate between left and right stabilizers and simply speak of the stabilizer of $G_0$, and when $G$ is written additively, we have that $\mathsf H (G_0) = \{ g \in G \t g + G_0 = G_0 \}$. Furthermore, for every $n \in \N$ and for a subgroup $H \subset G$, we denote by
\begin{itemize}
\item $[G : H]$ the {\it index} of $H$ in $G$,

\smallskip
\item $\phi_H : G \to G/H$ the canonical epimorphism if $H \subset G$ is normal,

\smallskip
\item $C_n$ an (additively written) {\it cyclic group} of order $n$,

\smallskip
\item $D_{2n}$ a {\it dihedral group} of order $2n$, and by

\smallskip
\item $Q_{4n}$ a {\it dicyclic group} of order $4n$.
\end{itemize}

\smallskip
\noindent
{\bf Sequences over groups.} Let $G$ be a finite group with identity element $1_G$ and $G_0 \subset G$ a subset. The elements of the free abelian monoid $\mathcal F (G_0)$ will be called  {\it sequences} over $G_0$.  This terminology goes back to Combinatorial Number Theory. Indeed, a sequence over $G_0$ can be viewed as a finite unordered sequence of terms from $G_0$, where the repetition of elements is allowed. We briefly discuss our notation which follows  the monograph \cite[Chapter 10.1]{Gr13a}. In order to avoid confusion between multiplication in $G$ and multiplication in $\mathcal F (G_0)$, we denote multiplication in $\mathcal F (G_0)$ by the boldsymbol $\bdot$ and we use brackets for all exponentiation in $\mathcal F (G_0)$. In particular, a sequence $S \in \mathcal F (G_0)$ has the form
\begin{equation} \label{basic}
S \, = \, g_1 \bdot \ldots \bdot g_{\ell} \, = \, {\small \prod}^{\bullet}_{i \in [1,\ell]} \, g_i \, \in \, \mathcal F (G_0),
\end{equation}
where $g_1, \ldots, g_{\ell} \in G_0$ are the terms of $S$. For $g \in G_0$,
\begin{itemize}
\item $\mathsf v_g (S) = | \{ i \in [1,\ell] \t g_i = g \} |$ denotes the {\it multiplicity} of $g$ in $S$,

\smallskip
\item $\supp(S) = \{ g \in G_0 \t \mathsf v_{g} (S) > 0 \}$ denotes the {\it support} of $S$, and

\smallskip
\item $\mathsf h (S) = \max \{ \mathsf v_g (S) \t g \in G_0 \}$ denotes the {\it maximal multiplicity} of $S$.
\end{itemize}
A {\it subsequence} $T$ of $S$ is a divisor of $S$ in $\mathcal F (G_0)$ and we write $T \t S$. For a subset $H \subset G_0$, we denote by $S_H$ the subsequence of $S$ consisting of all terms from $H$. Furthermore, $T \t S$ if and only if $\mathsf v_g (T) \le \mathsf v_g (S)$ for all $g \in G_0$, and in such case, $S \bdot T^{[-1]}$ denotes the subsequence of $S$ obtained by removing the terms of $T$ from $S$ so that $\mathsf v_g \big( S \bdot T^{[-1]} \big) = \mathsf v_g (S) - \mathsf v_g (T)$ for all $g \in G_0$. On the other hand, we set $S^{-1} = g^{-1}_1 \bdot \ldots \bdot g^{-1}_{\ell}$ to be the sequence obtained by taking elementwise inverse from $S$.

Moreover, if $S_1, S_2 \in \mathcal F (G_0)$ and $g_1, g_2 \in G_0$, then $S_1 \bdot S_2 \in \mathcal F (G_0)$ has length $|S_1|+|S_2|$, \ $S_1 \bdot g_1 \in \mathcal F (G_0)$ has length $|S_1|+1$, \ $g_1g_2 \in G$ is an element of $G$, but $g_1 \bdot g_2 \in \mathcal F (G_0)$ is a sequence of length $2$. If $g \in G_0$, $T \in \mathcal F (G_0)$, and $k \in \N_0$, then
\[
g^{[k]}=\underset{k}{\underbrace{g\bdot\ldots\bdot g}}\in \mathcal F (G_0) \quad \text{and} \quad T^{[k]}=\underset{k}{\underbrace{T\bdot\ldots\bdot T}}\in \mathcal F (G_0) \,.
\]
Let $S \in \mathcal F (G_0)$ be a sequence as in \eqref{basic}.
When $G$ is written multiplicatively, we denote by
\[
  \pi(S) \, = \, \{ g_{\tau (1)} \ldots  g_{\tau (\ell)} \in G \mid \tau \mbox{ a permutation of $[1, \ell]$} \} \, \subset \, G
\]
the {\it set of products} of $S$, and it can easily be seen that $\pi(S)$ is contained in a $G'$-coset, where $G'$ is the commutator subgroup of $G$. Note that $|S|=0$ if and only if $S = 1_{\mathcal F (G)}$, and in that case we use the convention that $\pi (S) = \{ 1_G \}$.
When $G$ is written additively with commutative operation, we likewise define
\[
  \sigma(S) \, = \, g_1 + \ldots + g_{\ell} \, \in \, G
\]
to be the {\it sum} of $S$. More generally, for any $n \in \N_0$, the {\it $n$-sums} and {\it $n$-products} of $S$ are respectively denoted by
\[
  \Sigma_n (S) \, = \, \{ \sigma(T) \mid T \text{ divides } S \und |T| = n \} \, \subset \, G \quad \und \quad \Pi_n (S) \, = \, \underset{|T|=n}{\bigcup_{T \t S}} \pi (T) \, \subset \, G \,,
\]
and the {\it subsequence sums} and {\it subsequence products} of $S$ are respectively denoted by
\[
  \Sigma (S) \, = \, \bigcup_{n \ge 1} \Sigma_n (S) \, \subset \, G \quad \und \quad \Pi (S) \, = \, \bigcup_{n \ge 1} \Pi_n (S) \, \subset \, G \,.
\]
The sequence $S$ is called
\begin{itemize}
\item a {\it product-one sequence} if $1_G \in \pi (S)$,

\smallskip
\item {\it product-one free} if $1_G \notin \Pi (S)$.

\smallskip
\item {\it square-free} if $\mathsf h (S) \le 1$.
\end{itemize}

\smallskip
\noindent
If $S = g_1 \bdot \ldots \bdot g_{\ell} \in \mathcal B (G)$ is a product-one sequence with $1_G = g_1 \ldots g_{\ell}$, then $1_G = g_i \ldots g_{\ell}g_1 \ldots g_{i-1}$ for every $i \in [1, \ell]$.
Every map of groups $\theta : G \rightarrow H$ extends to a monoid homomorphism $\theta : \mathcal F (G) \rightarrow \mathcal F (H)$, where $\theta (S) = \theta (g_1)\bdot \ldots \bdot \theta (g_{\ell})$.
If $\theta$ is a group homomorphism, then $\theta (S)$ is a product-one sequence if and only if $\pi(S) \cap \ker(\theta) \neq \emptyset$.
We denote by
\[
  \mathcal B (G_0) \, = \, \big\{ S \in \mathcal F (G_0) \t 1_G \in \pi(S) \big\}
\]
the set of all product-one sequences over $G_0$, and clearly $\mathcal B (G_0) \subset \mathcal F (G_0)$ is a submonoid. We denote by $\mathcal A (G_0)$ the set of irreducible elements of $\mathcal B (G_0)$ which, in other words, is the set of minimal product-one sequences over $G_0$.
Moreover,
\[
  \mathsf D (G_0) \, = \, \sup \big\{ |S| \t S \in \mathcal A (G_0) \big\} \, \in \, \N \cup \{ \infty \}
\]
is the {\it large Davenport constant} of $G_0$, and
\[
  \mathsf d (G_0) \, = \, \sup \big\{ |S| \t S \in \mathcal F (G_0) \mbox{ is product-one free } \big\} \, \in \, \N_0 \cup \{ \infty \}
\]
is the {\it small Davenport constant} of $G_0$.
It is well known that $\mathsf d (G) + 1 \le \mathsf D (G) \le |G|$, with equality in the first bound when $G$ is abelian, and equality in the second bound when $G$ is cyclic (\cite[Lemma 2.4]{Ge-Gr13a}). Moreover,  Geroldinger and  Grynkiewicz provide the precise value of the Davenport constants for non-cyclic groups having a cyclic index $2$ subgroups (see \cite{Ge-Gr13a, Gr13b}), whence we have that, for every $n \in \N_{\ge 2}$,
\[
  \mathsf D (Q_{4n}) \, = \, 3n \, \quad \und \quad \mathsf D (D_{2n}) \, = \, \left\{ \begin{array}{ll}
                                                                                         2n & \hbox{if $n \ge 3$ is odd} \\
                                                                               \frac{3n}{2} & \hbox{if $n \ge 4$ is even}.
                                                                                       \end{array} \right.
\]

\smallskip
\noindent
{\bf Ordered sequences over groups.} These are an important tool used to study (unordered) sequences over non-abelian groups. Indeed, it is quite useful to have related notation for sequences in which the order of terms matters. Thus, for a subset $G_0 \subset G$, we denote by $\mathcal F^{*} (G_0) = \big( \mathcal F^{*} (G_0), \bdot \big)$ the free (non-abelian) monoid with basis $G_0$, whose elements will be called the {\it ordered sequences} over $G_0$.

Taking an ordered sequence in $\mathcal F^{*} (G_0)$ and considering all possible permutations of its terms gives rise to a natural equivalence class in $\mathcal F^{*} (G_0)$, yielding a natural map
\[
  [ \bdot ] \, : \, \mathcal F^{*} (G_0) \quad \to \quad \mathcal F (G_0)
\]
given by abelianizing the sequence product in $\mathcal F^{*} (G_0)$. For any sequence $S \in \mathcal F (G_0)$, we say that an ordered sequence $S^{*} \in \mathcal F^{*} (G_0)$ with $[S^{*}] = S$ is an {\it ordering} of the sequence $S \in \mathcal F (G_0)$.

All notation and conventions for sequences extend naturally to ordered sequences. We sometimes associate an (unordered) sequence $S$ with a fixed (ordered) sequence having the same terms, also denoted by $S$. While somewhat informal, this does not give rise to confusion, and will improve the readability of some of the arguments.

For an ordered sequence $S = g_1 \bdot \ldots \bdot g_{\ell} \in \mathcal F^{*} (G)$, we denote by $\pi^{*} : \mathcal F^{*} (G) \to G$ the unique homomorphism that maps an ordered sequence onto its product in $G$, so
\[
  \pi^{*} (S) \, = \, g_1 \ldots g_{\ell} \, \in \, G \,.
\]
If $G$ is a multiplicatively written abelian group, then for every sequence $S \in \mathcal F (G)$, we always use $\pi^{*} (S) \in G$ to be the unique product, and $\Pi (S) = \bigcup \big\{ \pi^{*} (T) \t  T \text{ divides } S \und |T| \ge 1 \big\} \subset G$.

\smallskip
For the proof of our main results, the structure of product-one free sequences over  cyclic groups plays a crucial role. Thus we  gather some necessary lemmas regarding sequences over cyclic groups. Let $G$ be an additively written finite cyclic group.
A sequence $S \in \mathcal F (G)$ is called {\it smooth} (more precisely, {\it $g$-smooth}) if $S = (n_1 g) \bdot \ldots \bdot (n_{\ell} g)$, where $|S| = \ell \in \N$, $g \in G$, $1 = n_1 \le \ldots \le n_{\ell}$, $m = n_1 + \ldots + n_{\ell} < \ord (g)$, and $\Sigma (S) = \{ g, 2g, \ldots, mg \}$.

\smallskip
\begin{lemma} \cite[Lemma 5.1.4]{Ge09a} \label{2.1}~
Let $G$ be an additively written cyclic group of order $|G|=n \ge 3$, $g \in G$, and $k, l, n_1, \ldots, n_l \in \N$ such that $l \ge \frac{k}{2}$ and $m = n_1 + \ldots + n_l < k \le \ord (g)$. If $1 \le n_1 \le \ldots \le n_l$ and $S = (n_1 g) \bdot \ldots \bdot (n_l g)$, then $\sum (S) = \big\{ g, 2g, \ldots, mg \big\}$, and $S$ is $g$-smooth.
\end{lemma}

\smallskip
\begin{lemma} \label{2.2}~
Let $G$ be an additively written cyclic group of order $|G|=n \ge 3$ and $S \in \mathcal F (G)$ a product-one free sequence of length $|S| \ge \frac{n+1}{2}$. Then $S$ is $g$-smooth for some $g \in G$ with $\ord (g) = n$, and for every $h \in \sum (S)$, there exists a subsequence $T \t S$ such that $\sigma(T) = h$ and $g \t T$. In particular,
\begin{enumerate}
\item if $|S| = n - 1$, then $S = g^{[n-1]}$.

\smallskip
\item if $|S| = n - 2$, then $S = (2g) \bdot g^{[n-3]}$ or $S = g^{[n-2]}$.

\smallskip
\item if $n \ge 4$, then, for every subsequence $W \t S$ with $|W| \ge \frac{n}{2} - 1$, we obtain that $g \t W$.
\end{enumerate}
\end{lemma}

\begin{proof}
The first statement, that $S$ is $g$-smooth for some $g \in G$ with $\ord (g) = n$, was found independently by Savchev--Chen and by Yuan, and we cite it in the formulation of \cite[Theorem 5.1.8.1]{Ge09a}.

Suppose now that $S = (n_1 g) \bdot \ldots \bdot (n_{\ell} g)$ with $1 = n_1 \le \ldots \le n_{\ell}$. Then $n_2 + \ldots + n_{\ell} < n-1$ and $\ell - 1 \ge \frac{n-1}{2}$. Applying Lemma \ref{2.1} (with $k = n-1$), we obtain that $S \bdot g^{[-1]}$ is still $g$-smooth.
Let $h \in \sum (S) = \big\{ g, \, 2g,  \ldots, \, (n_1 + \ldots + n_{\ell})g \big\}$.
If $h = g$, then we take $T = g$. If $h \neq g$, then since $S \bdot g^{[-1]}$ is $g$-smooth, it follows that $h + (-g) \in \sum \big( S \bdot g^{[-1]} \big)$, and hence there exists $W \t S \bdot g^{[-1]}$ such that $\sigma(W) = h + (-g)$. Thus $W \bdot g$ is a subsequence of $S$ with $\sigma(W \bdot g) = h$.

\smallskip
1. and 2. This follows immediately from the main statement.

\smallskip
3. Let $n \ge 4$, and $W \t S$ be a subsequence with $|W| \ge \frac{n}{2} - 1$. Then there exists a subset $I \subset [1,\ell]$ with $|I| \ge \frac{n}{2} - 1$ such that $W = \prod^{\bullet}_{i \in I} (n_i g)$. Assume to the contrary that $n_i \ge 2$ for all $i \in I$.  Then
\[
  n - 1 \, \ge \, \sum_{j=1}^{\ell} n_j \, = \, \sum_{i \in I} n_i + \sum_{j \in [1,\ell] \setminus I} n_j \, \ge \, 2|W| + \big( |S| - |W| \big) \, = \, |S| + |W| \, \ge \, n - \frac{1}{2} \,,
\]
a contradiction.
\end{proof}


\bigskip
\section{On special sequences} \label{3}
\bigskip

\smallskip
In this section, we study the structure of minimal product-one sequences under certain additional conditions (Propositions \ref{3.2} and \ref{3.3}). These  results will be used substantially in the proofs of our main results in next section. We need the Theorem of DeVos--Goddyn--Mohar (see Theorem 13.1 of \cite{Gr13a} and the proceeding special cases).

\smallskip
\begin{lemma} \label{3.1}~
Let $G$ be a finite abelian group, $S\in \mathcal F(G)$ a sequence, $n \in [1,|S|]$, and $H = \mathsf H \big( \sum_n (S) \big)$. Then
\[
  |\Sigma_n (S)| \,\, \ge \,\, \left( \sum_{g \in G/H} \min \{n, \mathsf v_g \big( \phi_H (S) \big) \} - n + 1 \right)|H| \,.
\]
\end{lemma}

\smallskip
Let $G$ be an additively (resp. multiplicatively) written finite abelian group. Then $2G = \{ 2g \t g \in G \} \,\, \big( \textnormal{resp.} \,\, G^{2} = \{ g^{2} \t g \in G \} \big)$. Likewise, given a sequence $S = g_1 \bdot \ldots \bdot g_{\ell} \in \mathcal F (G)$, we set
\begin{equation} \label{double}
 2S \, = \, 2g_1 \bdot \ldots \bdot 2g_{\ell} \, \in \, \mathcal F (2G) \quad \big(\mbox{resp. } S^{2} \, = \, g^{2}_1 \bdot \ldots \bdot g^{2}_{\ell} \, \in \, \mathcal F \big( G^{2} \big) \big) \,.
\end{equation}
The {E}rd{\H{o}}s-{G}inzburg-{Z}iv constant $\mathsf s (G)$ is the smallest integer $\ell \in \N$ such that every sequence $S \in \mathcal F (G)$ of length $|S| \ge \ell$ has a subsequence $T \in \mathcal B (G)$ of length $|T| = \exp (G)$. If $G = C_{n_1} \oplus C_{n_2}$ with $1 \le n_1 \t n_2$, then $\mathsf s (G) = 2n_1+2n_2-3$ (\cite[Theorem 5.8.3]{Ge-HK06a}). Results on groups of higher rank can be found in \cite{Fa-Ga-Zh11a}.

\smallskip
\begin{proposition} \label{3.2}~
Let $G = \la \alpha, \tau \t \alpha^{n} = \tau^{2} = 1_G \und \tau\alpha = \alpha^{-1}\tau \ra$ be a dihedral group, where $n \in \N_{\ge 4}$ is even.
Let $S \in \mathcal F (G)$ be a minimal product-one sequence such that $|S| \ge n$ and $\supp(S) \subset G \setminus \la \alpha \ra$.
Then $S$ is a sequence of length $|S| = n$ having the following form{\rm \,:}
\begin{enumerate}
\item[(a)] If $n = 4$, then
           \[
             S \, = \, \tau \bdot \alpha\tau \bdot \alpha^{2}\tau \bdot \alpha^{3}\tau \, \quad \mbox{ or } \quad S \, = \, (\alpha^{x}\tau)^{[2]} \bdot \alpha^{y}\tau \bdot \alpha^{y+2}\tau \,,
           \]
           where $x, y \in [0,3]$ with $x \equiv y+1 \pmod{2}$.

\smallskip
\item[(b)] If $n \ge 6$, then
           \[
             S \, = \, (\alpha^{x}\tau)^{[v]} \bdot (\alpha^{\frac{n}{2} + x}\tau)^{[\frac{n}{2} - v]} \bdot (\alpha^{y}\tau)^{[w]} \bdot (\alpha^{\frac{n}{2} + y}\tau)^{[\frac{n}{2} - w]} \,,
           \]
           where $x, y \in [0, n-1]$ such that $2x \not\equiv 2y \pmod{n}$ and $\gcd(x-y, \frac{n}{2}) = 1$, and $v, w \in [0,\frac{n}{2}]$ such that $x-y \equiv v-w \pmod{2}$.
\end{enumerate}
In particular, there are no minimal product-one sequences $S$ over $G$ such that $S = S_1 \bdot S_2$ for some $S_1 \in \mathcal F \big( \la \alpha \ra \big)$ and $S_2 \in \mathcal F \big( G \setminus \la \alpha \ra \big)$ of length $|S_2| \ge n+2$.
\end{proposition}

\begin{proof}
For every $x \in \Z$, we set $\overline{x} = x + n \Z \in \Z/n\Z$. Let $S = \prod^{\bullet}_{i \in [1,|S|]} \alpha^{x_i}\tau \in \mathcal A (G)$ be of length $|S| \ge n$ with $\alpha^{x_1}\tau \ldots \alpha^{x_{|S|}}\tau = 1_G$, where $x_1, \ldots, x_{|S|} \in [0, n-1]$. Since $S \in \mathcal A (G)$, it follows that $|S|$ is even, and after renumbering if necessary, we set
\[
  W \, = \, \overline{x_1} \bdot \ldots \bdot \overline{x_{|S|}} \, = \, W_1 \bdot W_2 \, \in \, \mathcal F (\Z/n\Z) \,,
\]
where $W_1 = \prod^{\bullet}_{i \in [1,|S|/2]} \overline{x_{2i-1}}$ and $W_2 = \prod^{\bullet}_{i \in [1,|S|/2]} \overline{x_{2i}}$. Thus we have  $\sigma(W_1) = \sigma(W_2)$.
If we shift the sequence $W$ by $\overline{y}$ for some $y \in \Z$, then the corresponding sequence $S' = \prod^{\bullet}_{i \in [1,|S|]} \alpha^{x_i + y}\tau$ is still a minimal product-one sequence.
If $S'$ has the asserted structure, then the same is true for $S$ whence we may shift the sequence $W$ whenever this is convenient.
For every subsequence $U = \overline{y_1} \bdot \ldots \bdot \overline{y_{v}}$ of $W$, we denote by $\psi (U) = \alpha^{y_1}\tau \bdot \ldots \bdot \alpha^{y_v}\tau$ the corresponding subsequence of $S$.

\medskip
{\bf A1.} \, {\it Let $U = U_1 \bdot U_2$ be a subsequence of $W$ such that $|U_1| = |U_2|$ and $\sigma(U_1) = \sigma(U_2)$. Then $\psi (U)$ is a product-one sequence.}

\begin{proof}[Proof of {\bf A1}]
Suppose that $U_1 = \overline{y_1} \bdot \ldots \bdot \overline{y_{|U_1|}}$ and $U_2 = \overline{z_1} \bdot \ldots \bdot \overline{z_{|U_1|}}$. Since $\sigma(U_1) = \sigma(U_2)$, it follows that
\[
  \alpha^{y_1}\tau \alpha^{z_1}\tau \ldots \alpha^{y_{|U_1|}}\tau \alpha^{z_{|U_1|}}\tau \, = \, \alpha^{(y_1 + \ldots + y_{|U_1|}) - (z_1 + \ldots + z_{|U_1|})} \, = \, 1_G \,,
\]
whence $\psi (U)$ is a product-one sequence.
\end{proof}

\smallskip
If $\supp(W_1) \cap \supp(W_2) \neq \emptyset$, say $\overline{x_1} = \overline{x_2}$, then since $\sigma(W_1) = \sigma(W_2)$, it follows by {\bf A1} that $\psi( \overline{x_1} \bdot \overline{x_2})$ and $\psi \big( W \bdot ( \overline{x_1} \bdot \overline{x_2})^{[-1]} \big)$ are both product-one sequences, a contradiction.
Therefore $\supp(W_1) \cap \supp(W_2) = \emptyset$.

\medskip
\noindent
{\bf CASE 1.} \, $\mathsf h (W) = 1$.
\smallskip

Since $|W| \ge n = |\Z/n\Z|$, it follows that $|W| = n$, and hence $\supp(W) = \Z/n\Z$. Since $\sigma(W_1) = \sigma(W_2)$, it follows that
\[
  2(x_1 + x_3 + \ldots + x_{|S|-1}) \, \equiv \, \frac{n(n-1)}{2} \pmod{n} \,, \, \mbox{ whence } \, 2 \, \bigl\vert \, \frac{n}{2}(n-1) \,.
\]
Since $n$ is even, we have $\gcd(2, n-1) = 1$, which implies that $\frac{n}{2}$ is even. Note that, for any distinct two elements $x_{i_1}, x_{i_3} \in [1,\frac{n}{2}]$ with $x_{i_2} = x_{i_1} + \frac{n}{2}$ and $x_{i_4} = x_{i_3} + \frac{n}{2}$, the sequence $\prod^{\bullet}_{k \in [1,4]} \alpha^{x_{i_k}}\tau$ is a product-one sequence.
Since $\supp(W) = \Z/n\Z$, we have that $S$ is a product of $\frac{n}{4}$ product-one sequences of length $4$. Since $S \in \mathcal A (G)$, we must have that $n = 4$ and $W$ is a sequence over $\Z/4\Z$ with $\mathsf h(W) = 1$, whence $\psi(W)$ is the desired sequence for (a).

\medskip
\noindent
{\bf CASE 2.} \, $\mathsf h (W) \ge 2$.
\smallskip

Then there exists $i \in [1,|W|]$, say $i = 1$, such that $\mathsf v_{\overline{x_1}} (W) \ge 2$. In view of $\supp(W_1) \cap \supp(W_2) = \emptyset$, we may assume without loss of generality that $\overline{x_1} = \overline{x_3}$. Let
\[
  W' \, = \, \big( W_1 \bdot ( \overline{x_1} \bdot \overline{x_3})^{[-1]} \big) \bdot W_2 \quad \und \quad \ell \, = \, \frac{|W'|}{2} \, = \, \frac{|W|}{2} -1 \,.
\]

If $\sum_{\ell} (2W') = 2(\Z/n\Z)$, it follows by $\sigma(W') = 2\sigma(W_2) - 2 \overline{x_1} \in 2(\Z/n\Z)$ that there exits a subsequence $T \t W'$ of length $|T| = \ell$ such that $2\sigma(T) = \sigma(W')$. Hence we infer that $\sigma (T) = \sigma(W' \bdot T^{[-1]})$ and $|T| = |W' \bdot T^{[-1]}|$. Thus {\bf A1} implies that $\psi( \overline{x_1} \bdot \overline{x_3})$ and $\psi (W')$ are both product-one sequences, a contradiction. Therefore $\sum_{\ell} (2W') \subsetneq 2(\Z/n\Z)$.

Let $H = \mathsf H \big( \sum_{\ell} (2W') \big)$. By Lemma \ref{3.1}, we obtain that
\[
  |\Sigma_{\ell} (2W')| \,\, \ge \,\, \left( \sum_{g \in (2(\Z/n\Z))/H} \min \{\ell, \mathsf v_{g} \big( \phi_H (2W') \big) \} - \ell + 1 \right)|H| \,.
\]
If $\mathsf h \big( \phi_H (2W') \big) \le \ell$, then
\[
  |\Sigma_{\ell} (2W')| \,\, \ge \,\, \big( |2W'| - \ell + 1 \big)|H| \,\, \ge \,\, \frac{n}{2} \, = \, |2(\Z/n\Z)| \,,
\]
a contradiction. If there exist distinct $g_1, g_2 \in (2(\Z/n\Z))/H$ such that $\mathsf v_{g_k} \big( \phi_H (2W') \big) > \ell$ for all $k \in [1,2]$, then
\[
  |\Sigma_{\ell} (2W')| \,\, \ge \,\, (2\ell - \ell + 1)|H| \,\, \ge \,\, \frac{n}{2} \, = \, |2(\Z/n\Z)| \,,
\]
a contradiction. Thus there exists only one element, say $g \in (2(\Z/n\Z))/H$, such that $\mathsf v_{g} \big( \phi_H (2W') \big) > \ell$, which implies that
\[
  \mathsf v_{g} \big( \phi_H (2W') \big)  \ge |2W'|+1-\frac{|\Sigma_{\ell} (2W')|}{|H|} \ge |W'| + 2 - \frac{n}{2|H|} \,.
\]

\medskip
{\bf A2.} \, {\it If $H$ is trivial, then $|W| = n$ and $2W_2 = (2\overline{x_2})^{[\frac{n}{2}]}$ with $\mathsf v_{2\overline{x_2}} (2W_1) = 0$.}

\begin{proof}[Proof of {\bf A2}]
Suppose that $H$ is trivial. Then there exists $g \in 2(\Z/n\Z)$ such that  $\mathsf v_g (2W') \ge |W'| + 2 - \frac{n}{2} \ge \ell + 1$, and then we set $g = 2\overline{y}$ for some $y \in \Z$. If $\max \big\{ \mathsf v_{2\overline{y}} (2W_1), \mathsf v_{2\overline{y}} (2W_2) \big\} \le 1$, then $\ell + 1 \le \mathsf v_{2\overline{y}} (2W') \le \mathsf v_{2\overline{y}} (2W_1) + \mathsf v_{2\overline{y}} (2W_2) \le 2$, and thus $\ell \le 1$. Since $\ell \ge 1$, we obtain that $\ell = 1$, and it follows by $\ell = \frac{|W|}{2}-1$ that $|W| = n = 4$ and $|W_1| = |W_2| = 2$. Since $\max \big\{ \mathsf v_{2\overline{y}} (2W_1), \mathsf v_{2\overline{y}} (2W_2) \big\} \le 1$, we obtain that $2\overline{x_1} \neq 2\overline{y}$, and hence $2 = \ell+1 \le \mathsf v_{2\overline{y}} (2W') = \mathsf v_{2\overline{y}} (2W_2) \le 1$, a contradiction.
Thus we must have that $\max \big\{ \mathsf v_{2\overline{y}} (2W_1), \mathsf v_{2\overline{y}} (2W_2) \big\} \ge 2$, and assert that $\min \big\{ \mathsf v_{2\overline{y}} (2W_1), \mathsf v_{2\overline{y}} (2W_2) \big\} = 0$. Assume to the contrary that $\min \big\{ \mathsf v_{2\overline{y}} (2W_1), \mathsf v_{2\overline{y}} (2W_2) \big\} \ge 1$.
Then we may suppose by shifting if necessary that $2y \equiv 0 \pmod{n}$, and by symmetry that $\mathsf v_{2\overline{y}} (2W_1) \le \mathsf v_{2\overline{y}} (2W_2)$. Since $\supp(W_1) \cap \supp(W_2) = \emptyset$, we can assume that $\mathsf v_{\overline{y}} (W_1) = 0$ and $\mathsf v_{\overline{y}} (W_2) \ge 2$, and it follows that
\[
  \sigma \bigg( W_1 \bdot (\overline{y} + \overline{\frac{n}{2}})^{[-1]} \bigg) \, = \, \sigma \bigg( W_2 \bdot (\overline{y} + \overline{\frac{n}{2}}) \bdot (\overline{y} \bdot \overline{y})^{[-1]} \bigg) \,.
\]
Thus {\bf A1} ensures that $\psi(\overline{y} \bdot \overline{y})$ and $\psi(W \bdot (\overline{y} \bdot \overline{y})^{[-1]})$ are both product-one sequences, a contradiction. Hence $\min \big\{ \mathsf v_{2\overline{y}} (2W_1), \mathsf v_{2\overline{y}} (2W_2) \big\} = 0$, and it follows that
\[
  \ell + 1 \, \le \, \mathsf v_{2\overline{y}} (2W') \, = \, \max \big\{ \mathsf v_{2\overline{y}} (2(W_1\bdot(\overline{x_1} \bdot \overline{x_3})^{[-1]})), \mathsf v_{2\overline{y}} (2W_2) \big\} \, \le \, \ell + 1 \,.
\]
Thus $\mathsf v_{2\overline{y}} (2W') = \mathsf v_{2\overline{y}} (2W_2) = |W_2| = \ell + 1$. If $|W| \ge n+2$, then $\ell \ge \frac{n}{2}$, and thus $\mathsf v_{2\overline{y}} (2W') \ge |W'| + 2 - \frac{n}{2} \ge \ell + 2$, a contradiction. Therefore $|W| = n$ and $2W_2 =(2\overline{y})^{[\frac{n}{2}]}= (2\overline{x_2})^{[\frac{n}{2}]}$ with $\mathsf v_{2\overline{x_2}} (2W_1) = 0$.
\end{proof}

\smallskip
From now on, we assume that $(\overline{x_1}, \overline{x_3})$ is chosen to make $|H|$ maximal.

\medskip
\noindent
{\bf SUBCASE 2.1.} \, $H$ is non-trivial.
\smallskip

If $n = 4$, then $H \subset 2(\Z/4\Z)\cong C_2$ implies that $H = 2(\Z/4\Z)$, whence $\sum_{\ell} (2W') = 2(\Z/4\Z)$, a contradiction. Thus we can assume that $n \ge 6$.

Suppose that $[2(\Z/n\Z) : H] \ge 3$. Then $|H| \le \frac{n}{6}$, and since $\ell \ge \frac{n}{2} - 1$, we have
\[
  \mathsf v_g \big( \phi_H (2W') \big) \, \ge \, \ell + 1 + \frac{n}{2} - \frac{n}{2|H|} \, \ge \, \ell + 1 + 3|H| - 3 \,.
\]
Then it follows that $\min \big\{ \mathsf v_g \big( \phi_H (2W_1) \big), \mathsf v_g \big( \phi_H (2W_2) \big) \big\} \ge 3|H|-3$, for otherwise, we obtain that
\[
  \mathsf v_g \big( \phi_H (2W') \big) \, \le \, \mathsf v_g \big( \phi_H (2W_1) \big) + \mathsf v_g \big( \phi_H (2W_2) \big) \, \le \, (3|H| - 4) + \max \{|W_1|, |W_2|\} \, = \, \ell + 1 + 3|H|-4 \,,
\]
a contradiction. Moreover, we obtain that $\max \big\{ \mathsf v_g \big( \phi_H (2W_1) \big), \mathsf v_g \big( \phi_H (2W_2) \big) \big\} \ge 3|H| - 1$, for otherwise
\[
  \mathsf v_g \big( \phi_H (2W') \big) \, \le \, \mathsf v_g \big( \phi_H (2W_1) \big) + \mathsf v_g \big( \phi_H (2W_2) \big) \, \le \, 2(3|H| - 2) \, \le \frac{n}{2} + 3|H| - 4 \, \le \, \ell + 3|H| - 3 \,,
\]
a contradiction. Then it suffices to show the case when $\mathsf v_g \big( \phi_H (2W_1) \big) \le \mathsf v_g \big( \phi_H (2W_2) \big)$. Indeed the other case when $\mathsf v_g \big( \phi_H (2W_1) \big) \ge \mathsf v_g \big( \phi_H (2W_2) \big)$ follows by an identical argument.
Since $g \in (2(\Z/n\Z))/H$, by shifting if necessary, we can assume that $g = H$, whence $|(2W_1)_H | \ge 3|H|-3$ and $|(2W_2)_H| \ge 3|H| -1$.
Since $H$ is a non-trivial cyclic group, it follows by $\mathsf s (H) = 2|H|-1$ that there exist $U_1 \t W_1$ and $U_2 \t W_2$ such that $2U_1$ and $2U_2$ are zero-sum sequences over $H$ of length $|U_1| = |U_2| = |H|$.
Since $|\big( 2(W_2 \bdot U^{[-1]}_2)\big)_H| \ge 2|H|-1$, there also exists $U_3 \t W_2 \bdot U^{[-1]}_2$ such that $2U_3$ is a zero-sum sequence over $H$ of length $|U_3| = |H|$. Since $\sigma(U_k) \in \{ \overline{0}, \overline{\frac{n}{2}} \}$ for all $k \in [1,3]$, there exist distinct $i, j \in [1,3]$ such that $\sigma(U_i) = \sigma(U_j)$.
If $\sigma(U_1) = \sigma(U_j)$ for some $j \in [2,3]$, then $\sigma(W_1 \bdot U^{[-1]}_1) = \sigma(W_2 \bdot U^{[-1]}_j)$, and thus {\bf A1} implies that $\psi (U_1 \bdot U_j)$ and $\psi \big( W \bdot (U_1 \bdot U_j)^{[-1]}) \big)$ are both product-one sequences, a contradiction.
If $\sigma(U_2) = \sigma(U_3)$, then $\sigma \big( W_1 \bdot U^{[-1]}_1 \big) = \sigma \big( W_2 \bdot U_1 \bdot (U_2\bdot U_3)^{[-1]} \big)$ and $|W_1 \bdot U^{[-1]}_1| = \frac{|W|}{2} -|H| = |W_2 \bdot U_1 \bdot (U_2\bdot U_3)^{[-1]}|$. Thus {\bf A1} ensures that $\psi (U_2 \bdot U_3)$ and $\psi \big(W \bdot (U_2 \bdot U_3)^{[-1]}\big)$ are both product-one sequences, a contradiction.

Hence $[2(\Z/n\Z) : H] = 2$, and we obtain that $\mathsf v_g \big( \phi_H (2W') \big) \ge |W'|$. Then we may assume by shifting if necessary that $\supp(2W') \subset H$, and hence $\supp(W') \subset 2(\Z/n\Z)$. Since $\supp(W_1) \cap \supp(W_2) = \emptyset$ and $|W_2| \ge \frac{n}{2}$, we infer in view of $\supp(W_2) \subset 2(\Z/n\Z)$ that there exists $\overline{y} \in \supp(W_2)$ with $\mathsf v_{\overline{y}} (W_2) \ge 2$. By swapping the role between $(\overline{x_1}, \overline{x_3})$ and $(\overline{y}, \overline{y})$, we have that $|K| = |\mathsf H \big( \sum_{\ell} (2W'') \big)| \le |H|$ by the choice of $(\overline{x_1}, \overline{x_3})$, where $W'' = W_1 \bdot \big( W_2 \bdot (\overline{y} \bdot \overline{y})^{[-1]} \big)$. Then we assert that $2\overline{x_1} \in H$.
If $K$ is trivial, then {\bf A2} ensures that $2W_1 = (2\overline{x_1})^{[\frac{n}{2}]}$, and it follows by $n \ge 6$ that $2\overline{x_1} \in H$. If $K$ is non-trivial, then we must have $|K| = |H|$, for otherwise $[2(\Z/n\Z) : K] \ge 3$, and then the argument from the beginning of {\bf SUBCASE 2.1} leads to a contradiction. As two subgroups of a finite cyclic group having the same order are equal, we obtain that $K = H$, and since $W'$ and $W''$ share at least one term in common ($n \ge 6$), it follows that the $K$-coset containing $\supp(2W'')$ must be $H$, whence $2\overline{x_1} \in H$.
Thus, in all cases, we obtain that
\[
  \sigma(W') \, = \, 2\sigma(W_2) - 2\overline{x_1} \, \in \, H \, = \, \Sigma_{\ell} (2W') \,,
\]
where the final equality follows from the fact that $H$ is the stabilizer of $\sum_{\ell} (2W')$. Hence there exists $T \t W'$ of length $|T| = \ell$ such that $2\sigma(T) = \sigma(W')$, and thus we infer that $\sigma (T) = \sigma(W' \bdot T^{[-1]})$ and $|T| = |W' \bdot T^{[-1]}|$. Therefore {\bf A1} ensures that $\psi(\overline{x_1} \bdot \overline{x_3})$ and $\psi (W')$ are both product-one sequences, a contradiction.

\medskip
\noindent
{\bf SUBCASE 2.2.} \, $H$ is trivial.
\smallskip

By {\bf A2}, we have $2W_2 = (2\overline{x_2})^{[\frac{n}{2}]}$. If $\mathsf h (W_2) \ge 2$, then we may assume that $\overline{x_2} = \overline{x_4}$. By swapping the role between $(\overline{x_1}, \overline{x_3})$ and $(\overline{x_2}, \overline{x_4})$, it follows by the choice of $(\overline{x_1}, \overline{x_3})$ that $\mathsf H \big(\sum_{\ell} (2W'') \big)$ is also trivial, where $W'' = W_1 \bdot \big(W_2 \bdot (\overline{x_2} \bdot \overline{x_4})^{[-1]} \big)$. Again by {\bf A2}, we obtain that $2W_1 = (2\overline{x_1})^{[\frac{n}{2}]}$ with $2\overline{x_1} \neq 2\overline{x_2}$.

If $n = 4$, then we may assume in view of $\mathsf h (W) \ge 2$ that
\[
  W \, = \, W_1 \bdot W_2 \, = \, \overline{x_1}^{[2]} \bdot \Big( \overline{x_2} \bdot (\overline{x_2} + \overline{2}) \Big) \,,
\]
where $\overline{x_1}, \overline{x_2} \in \Z/4\Z$ with $2\overline{x_1} \neq 2\overline{x_2}$ (by {\bf A2}); Indeed, the other possibility is that $W = W_1 \bdot W_2 = \overline{x_1}^{[2]} \bdot \overline{x_2}^{[2]}$, which implies that $\psi(\overline{x_1} \bdot \overline{x_1})$ and $\psi(\overline{x_2} \bdot \overline{x_2})$ are both product-one sequences, a contradiction. Since $\sigma(W_1) = \sigma(W_2)$, it follows that $x_1 \equiv x_2 + 1 \pmod{2}$. Thus $\psi(W)$ is the desired sequence for (a).

If $n \ge 6$, then it follows by $2W_2 = (2\overline{x_2})^{[\frac{n}{2}]}$ and the Pigeonhole Principle that $\mathsf h (W_2) \ge 2$. Thus we obtain that $2W = (2\overline{x_1})^{[\frac{n}{2}]} \bdot (2\overline{x_2})^{[\frac{n}{2}]}$, whence
\[
  W \, = \, W_1 \bdot W_2 \, = \, \bigg( \overline{x_1}^{[v]} \bdot \big( \overline{x_1} + \overline{\frac{n}{2}} \big)^{[\frac{n}{2}-v]} \bigg) \bdot \bigg( \overline{x_2}^{[w]} \bdot \big( \overline{x_2} + \overline{\frac{n}{2}} \big)^{[\frac{n}{2}-w]} \bigg) \,,
\]
where $\overline{x_1}, \overline{x_2} \in \Z/n\Z$ with $2\overline{x_1} \neq 2\overline{x_2}$ (by {\bf A2}), and $v, w \in [0,\frac{n}{2}]$. Since $\sigma(W_1) = \sigma(W_2)$, it follows that $x_1 - x_2 \equiv v - w \pmod{2}$.
All that remains is to show that $\gcd (x_1 - x_2, \frac{n}{2}) = 1$. Assume to the contrary that $\gcd \big(x_1 - x_2, \frac{n}{2}\big) = d \ge 2$. Then we set $n' = \frac{n}{2d}$, and since $2W' = (2\overline{x_1})^{[\ell - 1]} \bdot (2\overline{x_2})^{[\ell + 1]}$, it follows by $n' (2 x_1 - 2 x_2) \equiv 0 \pmod{n}$ that
\[
  \Sigma_{\ell} (2W') \, = \, \big\{ k(2\overline{x_1} - 2\overline{x_2}) - 2\overline{x_2} \t k \in [0, n' - 1] \big\} \,.
\]
Thus we obtain that $2\overline{x_1} - 2\overline{x_2} \in \mathsf H \big(\sum_{\ell} (2W')\big) = H$, and since $H$ is trivial, it follows that $2\overline{x_1} = 2\overline{x_2}$, a contradiction. Therefore $\gcd(x_1 - x_2, \frac{n}{2}) = 1$.

\smallskip
To prove the ``In particular" statement, we assume to the contrary that there exists a minimal product-one sequence $S$ such that $S = S_1 \bdot S_2$, where $S_1 \in \mathcal F \big( \la \alpha \ra \big)$ and $S_2 \in \mathcal F \big( G \setminus \la \alpha \ra \big)$ with $|S_2| \ge n + 2$. Then we suppose that $S_2 = \prod^{\bullet}_{i \in [1,|S_2|]} \alpha^{x_i}\tau$ and $S_1 = T_1 \bdot T_2$ such that $\pi^{*} (T_1) (\alpha^{x_1}\tau) \pi^{*} (T_2) (\alpha^{x_2}\tau \ldots \alpha^{x_{|S_2|}}\tau) = 1_G$. Since $S \in \mathcal A (G)$, it follows that
\[
  S'' \, = \, \big( \pi^{*} (T_1)\alpha^{x_1}\tau \big) \bdot \big( \pi^{*} (T_2)\alpha^{x_2}\tau \big) \bdot \Big( {\small\prod}^{\bullet}_{i \in [3,|S_2|]} \alpha^{x_i}\tau \Big) \, \in \, \mathcal A \big( G \setminus \la \alpha \ra \big)
\]
of length $|S''| = |S_2| \ge n + 2$, but this is impossible by the main statement.
\end{proof}

\smallskip
\begin{proposition} \label{3.3}~
Let $G = \la \alpha, \tau \t \alpha^{2n} = 1_G, \, \tau^{2} = \alpha^{n}, \und \tau\alpha = \alpha^{-1}\tau \ra$ be a dicyclic group, where $n \ge 2$. Let $S \in \mathcal F (G)$ be a minimal product-one sequence such that $|S| \ge 2n + 2$ and $\supp(S) \subset G \setminus \la \alpha \ra$. Then $S$ is a sequence of length $|S| = 2n+2$ having the form
\[
  S \, = \, (\alpha^{x}\tau)^{[n+2]} \, \bdot \, S_0 \,,
\]
where $x \in [0, 2n-1]$, and $S_0$ is a sequence of length $|S_0| = n$ having one of the following two forms{\rm \,:}
\begin{enumerate}
\item[(a)] $S_0 = (\alpha^{y}\tau)^{[2]} \bdot \alpha^{y + n}\tau \bdot \alpha^{y_1}\tau \bdot \ldots \bdot \alpha^{y_{n-3}}\tau$, where $n \ge 3$, $y, y_1, \ldots, y_{n-3} \in [0, 2n-1]$ such that $2y \not\equiv 2x \pmod{2n}$, $2y_i \not\equiv 2x \pmod{2n}$ for all $i$, and $(y_1 + \ldots + y_{n-3}) + 3y + n + x \equiv (n+1)(x + n) \pmod{2n}$

\smallskip
\item[(b)] $S_0 = (\alpha^{y}\tau)^{[n]}$, where $y \in [0, 2n-1]$ such that $2y \not\equiv 2x \pmod{2n}$ and $ny + x \equiv (n+1)(x+n) \pmod{2n}$.
\end{enumerate}
In particular, there are no minimal product-one sequences $S$ over $G$ such that $S = S_1 \bdot S_2$ for some $S_1 \in \mathcal F \big( \la \alpha \ra \big)$ and $S_2 \in \mathcal F \big( G \setminus \la \alpha \ra \big)$ of length $|S_2| \ge 2n + 4$.
\end{proposition}

\begin{proof}
For every $x \in \Z$, we set $\overline{x} = x + 2n \Z \in \Z/2n\Z$.
Let $S = \prod^{\bullet}_{i \in [1,|S|]} \alpha^{x_i}\tau \in \mathcal A (G)$ be of length $|S| \ge 2n + 2$ with $\alpha^{x_1}\tau \ldots \alpha^{x_{|S|}}\tau = 1_G$, where $x_1, \ldots, x_{|S|} \in [0, 2n-1]$.
Since $S \in \mathcal A (G)$, it follows that $|S|$ is even, and after renumbering if necessary, we set
\[
  W \, = \, \overline{x_1} \bdot \ldots \bdot \overline{x_{|S|}} \, = \, W_1 \bdot W_2 \, \in \, \mathcal F (\Z/2n\Z) \,,
\]
where $W_1 = \prod^{\bullet}_{i \in [1,|S|/2]} \overline{x_{2i-1}}$, and $W_2 = \prod^{\bullet}_{i \in [1,|S|/2]} \overline{x_{2i}}$. Thus we have that $\sigma(W_1) \, = \, \sigma(W_2) + |W_1|\overline{n}$. If we shift the sequence $W$ by $\overline{y}$ for some $y \in \Z$, then the corresponding sequence $S' = \prod^{\bullet}_{i \in [1,|S|]} \alpha^{x_i + y}\tau$ is still a minimal product-one sequence.
If $S'$ has the asserted structure, then the same is true for $S$ whence we may shift the sequence $W$ whenever this is convenient.
For every subsequence $U = \overline{y_1} \bdot \ldots \bdot \overline{y_{v}}$ of $W$, we denote by $\psi (U) = \alpha^{y_1}\tau \bdot \ldots \bdot \alpha^{y_v}\tau$ the corresponding subsequence of $S$.

\medskip
{\bf A1.} \, {\it Let $U = U_1 \bdot U_2$ be a subsequence of $W$ such that $|U_1| = |U_2|$ and $\sigma(U_1) = \sigma(U_2) + |U_1|\overline{n}$. Then $\psi (U)$ is a product-one sequence.}

\begin{proof}[Proof of {\bf A1}]
Suppose that $U_1 = \overline{y_1} \bdot \ldots \bdot \overline{y_{|U_1|}}$ and $U_2 = \overline{z_1} \bdot \ldots \bdot \overline{z_{|U_1|}}$. Since $\sigma(U_1) = \sigma(U_2) + |U_1|\overline{n}$, it follows that
\[
  \alpha^{z_1}\tau \alpha^{y_1}\tau \ldots \alpha^{z_{|U_1|}}\tau \alpha^{y_{|U_1|}}\tau \, = \, \alpha^{(z_1 + \ldots + z_{|U_1|}) - (y_1 + \ldots + y_{|U_1|}) + |U_1|n} \, = \, 1_G \,,
\]
whence $\psi (U)$ is a product-one sequence.
\end{proof}

\smallskip
If $\supp(W_1)\cap \big( \supp(W_2) + \overline{n} \big) \neq \emptyset$, say $\overline{x_1} = \overline{x_2} + \overline{n}$, then since $\sigma(W_1) = \sigma(W_2) + |W_1| \overline{n}$, it follows by {\bf A1} that $\psi(\overline{x_1} \bdot \overline{x_2})$ and $\psi \big( W \bdot (\overline{x_1} \bdot \overline{x_2})^{[-1]} \big)$ are both product-one sequences, a contradiction. Therefore $\supp(W_1) \cap \big( \supp(W_2) + \overline{n} \big) = \emptyset$, and since $|S| \ge 2n + 2$, it follows that $\mathsf h (W) \ge 2$.

\medskip
{\bf A2.} \, {\it $\min \big\{ \mathsf v_{2\overline{g}} (2W_1), \mathsf v_{2\overline{g}} (2W_2) \big\} \le 1$ for every $\overline{g} \in \Z/2n\Z$.}

\begin{proof}[Proof of {\bf A2}]
Assume to the contrary that there exists $\overline{g} \in \Z/2n\Z$ such that $\min \big\{ \mathsf v_{2\overline{g}} (2W_1), \mathsf v_{2\overline{g}} (2W_2) \big\} \ge 2$. Then, for each $i \in [1,2]$, we have $\mathsf v_{\overline{g}} (W_i) + \mathsf v_{\overline{g} + \overline{n}} (W_i) = \mathsf v_{2\overline{g}} (2W_i) \ge 2$. We may assume without loss of generality that $\mathsf v_{\overline{g}} (W_1) \ge 1$. Since $\supp(W_1) \cap \big( \supp(W_2) + \overline{n} \big) = \emptyset$, we must have $\mathsf v_{\overline{g} + \overline{n}} (W_2) = 0$, whence $\mathsf v_{\overline{g}} (W_2) \ge 2$. Since $\supp(W_1) \cap \big( \supp(W_2) + \overline{n} \big) = \emptyset$, we must have $\mathsf v_{\overline{g} + \overline{n}} (W_1) = 0$, whence $\mathsf v_{\overline{g}} (W_1) \ge 2$.
We set $U_1 = U_2 = \overline{g} \bdot \overline{g}$. It follows that $U_1 \t W_1$ and $U_2 \t W_2$ such that $|U_1| = |U_2|$ with $\sigma(U_1) = \sigma(U_2) + |U_1| \overline{n}$, and $|W_1 \bdot U^{[-1]}_1| = |W_2 \bdot U^{[-1]}_2|$ with $\sigma \big( W_1 \bdot U^{[-1]}_1 \big) = \sigma \big( W_2 \bdot U^{[-1]}_2 \big) + |W_1 \bdot U^{[-1]}_1| \overline{n}$. Thus {\bf A1} ensures that $\psi(U_1 \bdot U_2)$ and $\psi \big(W \bdot (U_1 \bdot U_2)^{[-1]} \big)$ are both product-one sequences, a contradiction.
\end{proof}

\medskip
\noindent
{\bf CASE 1.} \, There exists $\overline{y} \in \supp(W)$ such that $\mathsf v_{\overline{y}} (W) \ge 2$ and  $\overline{y} + \overline{n} \in \supp(W)$.
\smallskip

In view of $\supp(W_1) \cap \big(\supp(W_2) + \overline{n} \big) = \emptyset$, we may assume without loss of generality that $\overline{y} \bdot (\overline{y} + \overline{n}) \t W_1$. Let
\[
  W' \, = \, W \bdot \big( \overline{y} \bdot (\overline{y} + \overline{n}) \big)^{[-1]} \quad \und \quad \ell \, = \, \frac{|W'|}{2} \, = \, \frac{|W|}{2} - 1 \,.
\]

If $\sum_{\ell} (2W') = 2(\Z/2n\Z)$, then since $\sigma(W') + \ell \overline{n} = 2\sigma(W_2) + 2\ell \overline{n} - 2\overline{y} \in 2(\Z/2n\Z)$, it follows that there exits a subsequence $T \t W'$ of length $|T| = \ell$ such that $2 \sigma(T) = \sigma(W') + \ell \overline{n}$. Hence we infer that $\sigma(T) = \sigma(W' \bdot T^{[-1]}) + |T|\overline{n}$ and $|T| = |W' \bdot T^{[-1]}|$. Thus {\bf A1} ensures that $\psi \big(\overline{y} \bdot (\overline{y} + \overline{n}) \big)$ and $\psi (W')$ are both product-one sequences, a contradiction. Therefore $\sum_{\ell} (2W') \subsetneq 2(\Z/2n\Z)$.

Let $H = \mathsf H \big( \sum_{\ell} (2W') \big)$. By Lemma \ref{3.1}, we obtain that
\[
  |\Sigma_{\ell} (2W')| \,\, \ge \,\, \left( \sum_{g \in (2(\Z/2n\Z))/H} \min \{\ell, \mathsf v_g \big( \phi_H (2W') \big) \} - \ell + 1 \right)|H| \,.
\]
If $\mathsf h \big( \phi_H (2W') \big) \le \ell$, then
\[
  |\Sigma_{\ell} (2W')| \, \ge \, \big( |2W'| - \ell + 1 \big)|H| \, \ge \, n \, = \, |2(\Z/2n\Z)| \,,
\]
a contradiction.
If there exist distinct $g_1, g_2 \in (2(\Z/2n\Z))/H$ such that $\mathsf v_{g_k} \big( \phi_H (2W') \big) > \ell$ for all $k \in [1,2]$, then
\[
  |\Sigma_{\ell} (2W')| \, \ge \, (2\ell - \ell + 1)|H| \, \ge \, n \,= \, |2(\Z/2n\Z)| \,,
\]
a contradiction.
Thus there exists only one element, say $g \in (2(\Z/2n\Z))/H$, such that $\mathsf v_g \big( \phi_H (2W') \big) > \ell$, which implies that
\[
  \mathsf v_g \big( \phi_H (2W') \big) \, \ge \, |2W'| + 1 - \frac{|\Sigma_{\ell} (2W')|}{|H|} \, \ge \, |W'| + 2 - \frac{n}{|H|} \,.
\]

\medskip
\noindent
{\bf SUBCASE 1.1.} $H$ is non-trivial.
\smallskip

If $[2(\Z/2n\Z) : H] = 2$, then $\mathsf v_g \big( \phi_H (2W') \big) \ge |W'|$. We may assume by shifting if necessary that $\supp(2W') \subset H$, and hence $\supp(W') \subset 2(\Z/2n\Z)$. Since $\mathsf v_{\overline{y}} (W) \ge 2$, it follows that $\overline{y} \in \supp(W') \subset 2(\Z/2n\Z)$, whence $\sigma(W') + \ell \overline{n} = 2\sigma(W_2) - 2 \overline{y} \in H$. Thus there exists $T \t W'$ of length $|T| = \ell$ such that $2\sigma(T) = \sigma(W') + \ell \overline{n}$, and hence we infer that $\sigma(T) = \sigma(W' \bdot T^{[-1]}) + |T|\overline{n}$ and $|T| = |W' \bdot T^{[-1]}|$. It follows by {\bf A1} that $\psi \big( \overline{y} \bdot (\overline{y} + \overline{n}) \big)$ and $\psi (W')$ are both product-one sequences, a contradiction.

Therefore $[2(\Z/2n\Z) : H] \ge 3$, and hence $|H| \le \frac{n}{3}$. Since $\ell \ge n$, we have
\[
  \mathsf v_g \big( \phi_H (2W') \big) \, \ge \, \ell + 1 + (n + 1) - \frac{n}{|H|} \, \ge \, \ell + 2 + 3|H| - 3 \,.
\]
Then  $\min \big\{ \mathsf v_g \big( \phi_H (2W_1) \big), \mathsf v_g \big( \phi_H (2W_2) \big) \big\} \ge 3|H| - 2$, for otherwise, we obtain that
\[
  \mathsf v_g \big( \phi_H (2W') \big) \, \le \, \mathsf v_g \big( \phi_H (2W_1) \big) + \mathsf v_g \big( \phi_H (2W_2) \big) \, \le \,3|H| - 3 + \max \big\{ |W_1|, |W_2| \big\} \, \le \, 3|H| - 3 + \ell + 1 \,,
\]
a contradiction. Since $g \in (2(\Z/2n\Z))/H$, by shifting if necessary, we can assume that $g = H$, whence $|(2W_i)_H| \ge 3|H| - 2$ for all $i \in [1,2]$. It follows by $\mathsf s (H) = 2|H| - 1$ that there exist $U_1 \t W_1$ and $V_1 \t W_2$ of length $|U_1| = |V_1| = |H|$ such that $\sigma(U_1), \sigma(V_1) \in \{ \overline{0}, \overline{n} \}$. Therefore $|\big( 2W_1 \bdot (2U_1)^{[-1]}\big)_H| \ge 2|H| - 2$ and $|\big( 2W_2 \bdot (2V_1)^{[-1]}\big)_H| \ge 2|H| - 2$.

Suppose that there exist $U_2 \t W_1 \bdot U_1^{[-1]}$ and $V_2 \t W_2 \bdot V^{[-1]}_1$ with $|U_2| = |V_2| = |H|$ and $\sigma(U_2), \sigma(V_2) \in \{ \overline{0}, \overline{n} \}$. If there exits $i\in [1,2]$ such that $\sigma(U_i) = \sigma(V_i) + |H| \overline{n}$, then {\bf A1} implies that $\psi(U_i \bdot V_i)$ and $\psi \big(W \bdot (U_i \bdot V_i)^{[-1]} \big)$ are both product-one sequences, a contradiction. Otherwise, we have $\sigma(U_1 \bdot U_2) = \sigma(V_1 \bdot V_2) + 2|H| \overline{n}$, whence {\bf A1} ensures that $\psi(U_1 \bdot U_2 \bdot V_1 \bdot V_2)$ and $\psi \big( W \bdot (U_1 \bdot U_2 \bdot V_1 \bdot V_2)^{[-1]} \big)$ are both product-one sequences, a contradiction.

Assume that either $\big( 2W_1 \bdot (2U_1)^{[-1]} \big)_H$ or $\big( 2W_2 \bdot (2V_1)^{[-1]} \big)_H$ dose not contain a zero-sum subsequence of length $|H|$, say $2W_1 \bdot (2U_1)^{[-1]}$, which then forces $| \big( 2W_1 \bdot (2U_1)^{[-1]} \big)_H | = 2|H| - 2$. By \cite[Proposition 5.1.12]{Ge09a}, there exist $h_1, h_2 \in H$ with $\ord(h_1 - h_2) = |H|$ such that $\Big( 2W_1 \bdot (2U_1)^{[-1]} \Big)_H = h^{[|H|-1]}_1 \bdot h^{[|H|-1]}_2$. Then $\ord(h_1 - h_2) = |H|$ ensures that
\[
  H \, = \, \underset{|H|-1}{\underbrace{\{ h_1, h_2 \} + \ldots + \{ h_1, h_2 \}}} \, = \, \Sigma_{|H|-1} \big( h^{[|H|-1]}_1 \bdot h^{[|H|-1]}_2 \big) \,= \, \Sigma_{|H|-1} \big( (2W_1\bdot (2U_1)^{[-1]})_H \big) \,.
\]
Thus we infer that there exist subsequences $2U_3 \t 2W_1\bdot (2U_1)^{[-1]}$ and $2V_3 \t 2W_2 \bdot (2V_1)^{[-1]}$ such that $|2U_3| = |2V_3| = |H| - 1$ and $\sigma(2U_3) = \sigma(2V_3)$. Hence $\sigma(U_3) = \sigma(V_3)$ or $\sigma(U_3) = \sigma(V_3) + \overline{n}$. If there exits $i \in \{1,3\}$ such that $\sigma(U_i) = \sigma(V_i) + |U_i| \overline{n}$, then {\bf A1} implies that $\psi(U_i \bdot V_i)$ and $\psi \big(W \bdot (U_i \bdot V_i)^{[-1]} \big)$ are both product-one sequences, a contradiction. Otherwise, we have $\sigma(U_1 \bdot U_3) = \sigma(V_1 \bdot V_3) + (2|H|-1) \overline{n}$, whence {\bf A1} ensures that $\psi(U_1 \bdot U_3 \bdot V_1 \bdot V_3)$ and $\psi \big( W \bdot (U_1 \bdot U_3 \bdot V_1 \bdot V_3)^{[-1]} \big)$ are both product-one sequences, a contradiction.

\medskip
\noindent
{\bf SUBCASE 1.2.} \, $H$ is trivial.
\smallskip

Since $\ell = \frac{|W'|}{2} \ge n$, it follows that $\mathsf v_g (2W') \ge |W'| + 2 - n \ge \ell + 2$. Hence {\bf A2} ensures that $\min \big\{ \mathsf v_g (2W_1), \mathsf v_g (2W_2) \big\} = 1$. If $g = 2\overline{y}$, it follows by $\overline{y} \bdot (\overline{y} + \overline{n}) \t W_1$ that $\mathsf v_g (2W_2) = 1$, whence $\ell + 2 \le \mathsf v_g (2W') = \mathsf v_g (2W_1) -2 + 1 \le \ell$, a contradiction. Thus $g \neq 2\overline{y}$. Since
\[
  \ell + 2 \, \le \, \mathsf v_g (2W') \, = \, \mathsf v_g \Big( 2 \big(W_1 \bdot (\overline{y} \bdot (\overline{y} + \overline{n}))^{[-1]} \big) \Big) + \mathsf v_g (2W_2) \,,
\]
we have $\mathsf v_g (2W_1) = 1$ and $\mathsf v_g (2W_2) = \ell + 1$.  Then $\mathsf v_g (2W') = \ell + 2$. If $|W| \ge 2n + 4$, then $\ell \ge n + 1$, and hence $\mathsf v_g (2W') \ge |W'| + 2 - n \ge \ell + 3$, a contradiction. Therefore $|W| = 2n + 2$, $\ell = n$, $2W_2 = (2\overline{x})^{[n+1]}$, and $\mathsf v_{2\overline{x}} (2W_1) = 1$ for some $\overline{x} \in \Z/2n\Z$ with $2\overline{x} = g \neq 2\overline{y}$.

Since $\supp(W_1) \cap \big( \supp(W_2) + \overline{n} \big) = \emptyset$, we may assume that $W_2 = \overline{x}^{[n+1]}$ and $\mathsf v_{\overline{x}} (W_1) = 1$. It follows by $\mathsf v_{\overline{y}} (W) \ge 2$ and $|W_1| = n + 1$ that $\overline{x} \bdot \overline{y} \bdot \overline{y} \bdot (\overline{y} + \overline{n}) \t W_1$. Then $n \ge 3$ and
\[
  W \, = \, W_1 \bdot W_2 \, = \, (\overline{x} \bdot T) \bdot \overline{x}^{[n+1]} \,,
\]
where $T \in \mathcal F \big( \Z/2n\Z \big)$ with $|T| = n$ such that $2\overline{x} \notin \supp(2T)$ and $\overline{y}^{[2]} \bdot (\overline{y} + \overline{n}) \t T$. Since $\sigma(W_1) = \sigma(W_2) + |W_1| \overline{n}$, it follows that $\sigma(T) + \overline{x} \, = \, (n + 1)\overline{x} + (n + 1)\overline{n}$. Therefore $\psi(W)$ is the desired sequence for (a).

\medskip
\noindent
{\bf CASE 2.} \, For every $\overline{x} \in \supp(W)$ with $\mathsf v_{\overline{x}} (W) \ge 2$, we have that $\overline{x} + \overline{n} \notin \supp(W)$.
\smallskip

If $\mathsf h (2W) \le 2$, then we have
\[
  2n + 2 \, \le \, |W| \, = \, |2W| \, \le \, \mathsf h (2W) |2(\Z/2n\Z)| \, \le \, 2n \,,
\]
a contradiction, and from the case hypothesis, we have  $\mathsf h (W) = \mathsf h (2W) \ge 3$. Let $\overline{x} \in \supp(W)$ be an element with $\mathsf v_{\overline{x}} (W) = \mathsf h (W) \ge 3$, and assume without loss of generality that
\[
  \mathsf v_{\overline{x}} (W_1) \, \ge \, \mathsf v_{\overline{x}} (W_2) \quad \mbox{ with } \,\, \mathsf v_{\overline{x}} (W_1) \, \ge \, 2 \,.
\]
If $\mathsf h (W \bdot (\overline{x} \bdot \overline{x})^{[-1]}) \le 1$, then it follows by the case hypothesis that
\[
  2n \, \le \, |W| - 2 \, = \, |W \bdot (\overline{x} \bdot \overline{x})^{[-1]}| \, \le \, |(\Z/2n\Z) \setminus \{ \overline{x} + \overline{n} \}| \, = \, 2n - 1 \,,
\]
a contradiction, whence $\mathsf h (W \bdot (\overline{x} \bdot \overline{x})^{[-1]}) \ge 2$. Let $\overline{y} \in \supp(W \bdot (\overline{x} \bdot \overline{x})^{[-1]})$ be an element with $\mathsf v_{\overline{y}} (W \bdot (\overline{x} \bdot \overline{x})^{[-1]}) \ge 2$, and let
\[
  W' \, = \, W \bdot (\overline{x} \bdot \overline{x} \bdot \overline{y} \bdot \overline{y})^{[-1]} \quad \und \quad \ell \, = \, \frac{|W'|}{2} \, = \, \frac{|W|}{2} - 2 \,.
\]
Suppose in addition that $\overline{y}$ is chosen to satisfy either that $\mathsf v_{\overline{y}} (W \bdot (\overline{x} \bdot \overline{x})^{[-1]}) = \mathsf h (W \bdot (\overline{x} \bdot \overline{x})^{[-1]})$, or that both $\mathsf v_{\overline{y}} (W_2) \ge 3$ and $\mathsf h (W) \le \ell + 2$.

If $\sum_{\ell} (2W') = 2(\Z/2n\Z)$, then since $\sigma(W') + \ell \overline{n} = 2\sigma(W_2) + (2\ell + 2)\overline{n} - 2\overline{x} - 2\overline{y} \in 2(\Z/2n\Z)$, it follows that there exists a subsequence $T \t W'$ of length $|T| = \ell$ such that $2\sigma(T) = \sigma(W') + \ell \overline{n}$. Hence we infer $\sigma(T) = \sigma(W' \bdot T^{[-1]}) + |T|\overline{n}$ and $|T| = |W' \bdot T^{[-1]}|$. Thus {\bf A1} ensures that $\psi \big( \overline{x}^{[2]} \bdot \overline{y}^{[2]} \big)$ and $\psi (W')$ are both product-one sequences, a contradiction. Therefore $\sum_{\ell} (2W') \subsetneq 2(\Z/2n\Z)$.

Let $H = \mathsf H \big( \sum_{\ell} (2W') \big)$. As at the start of the proof of {\bf CASE 1}, it follows by Lemma \ref{3.1} that there exists only one element, say $g \in (2(\Z/2n\Z))/H$, such that $\mathsf v_{g} \big( \phi_H (2W') \big) \ge \ell + 1$, which implies that
\[
  \mathsf v_g \big( \phi_H (2W') \big) \, \ge \, |2W'| + 1 - \frac{|\Sigma_{\ell} (2W')|}{|H|} \, \ge \, |W'| + 2 - \frac{n}{|H|} \,.
\]

\medskip
\noindent
{\bf SUBCASE 2.1.} \, $H$ is non-trivial.
\smallskip

If $n = 2$, then $H \subset 2(\Z/4\Z)\cong C_2$ implies that $H = 2(\Z/4\Z)$, whence $\sum_{\ell} (2W') = 2(\Z/4\Z)$, a contradiction. Thus we can assume that $n \ge 3$.

If $[2(\Z/2n\Z) : H] = 2$, then $\mathsf v_g \big( \phi_H (2W') \big) \ge |W'|$. We may assume by shifting if necessary that $\supp(2W') \subset H$, and hence $\supp(W') \subset 2(\Z/2n\Z)$.
We assert that $\sigma(W') + \ell \overline{n} = 2\sigma(W_2) - 2 \overline{x} - 2 \overline{y} \in H$. Clearly this holds true for $\overline{x} = \overline{y}$. Suppose  $\overline{x} \neq \overline{y}$. Since $\mathsf v_{\overline{x}} (W) = \mathsf h (W) \ge 3$, it follows that $\overline{x} \in \supp(W') \subset 2(\Z/2n\Z)$. If $\mathsf v_{\overline{y}} (W_2) \ge 3$, then $\overline{y} \in \supp(W') \subset 2(\Z/2n\Z)$. Suppose that $\mathsf v_{\overline{y}} (W \bdot (\overline{x} \bdot \overline{x})^{[-1]}) = \mathsf h (W \bdot (\overline{x} \bdot \overline{x})^{[-1]})$, and we need to verify $\overline{y} \in \supp(W') \subset 2(\Z/2n\Z)$. If $\mathsf h (2W') \le 2$, then
\[
  2n - 2 \, \le \, |W'| \, = \, |2W'| \, \le \, \mathsf h (2W') |H| \, \le \, n \,,
\]
a contradiction to $n\ge 3$. Hence, in view of the main case hypothesis, we have $\mathsf h (W') = \mathsf h (2W') \ge 3$. Since $\mathsf h (W \bdot (\overline{x} \bdot \overline{x})^{[-1]}) \ge \mathsf h (W') \ge 3$, it follows that $\overline{y} \in \supp(W') \subset 2(\Z/2n\Z)$. Thus $\sigma(W') + \ell \overline{n} \in H$, which implies that there exists a subsequence $T \t W'$ of length $|T| = \ell$ such that $2\sigma(T) = \sigma(W') + \ell \overline{n}$. Then $\sigma(T) = \sigma(W' \bdot T^{[-1]}) + |T|\overline{n}$ and $|T| = |W' \bdot T^{[-1]}|$. It follows by {\bf A1} that $\psi \big( \overline{x}^{[2]} \bdot \overline{y}^{[2]} \big)$ and $\psi (W')$ are both product-one sequences, a contradiction.

Therefore $[2(\Z/2n\Z) : H] \ge 3$, and hence $|H| \le \frac{n}{3}$. Since $\ell = \frac{|W'|}{2} \ge n - 1$, we have
\[
  \mathsf v_g \big( \phi_H (2W') \big) \, \ge \, \ell+1 + n - \frac{n}{|H|} \, \ge \, \ell + 1 + 3|H| - 3 \,.
\]
We assert that $\min \big\{ \mathsf v_g \big( \phi_H (2W_1) \big), \mathsf v_g \big( \phi_H (2W_2) \big) \big\} \ge 3|H| - 2$. Assume to the contrary that
\[
  \min \big\{ \mathsf v_g \big( \phi_H (2W_1) \big), \mathsf v_g \big( \phi_H (2W_2) \big) \big\} \, \le \, 3|H| - 3 \,.
\]
If $\mathsf v_g \big( \phi_H (2W_2) \big) \le \ell$, then $\mathsf v_{\overline{x}} (W_1) \ge 2$ implies that
\[
  \mathsf v_g \big( \phi_H (2W') \big) \, \le \, \mathsf v_g \Big( \phi_H \big(2(W_1 \bdot (\overline{x} \bdot \overline{x})^{[-1]}) \big) \Big) + \mathsf v_g \big( \phi_H (2W_2) \big) \, \le \, \ell + 3|H|-3 \,,
\]
a contradiction. Thus $\mathsf v_g \big( \phi_H (2W_2) \big) \ge \ell+1 \ge n$, and hence $\mathsf h (2W_2) \ge \frac{n}{|H|} \ge 3$. The main case hypothesis ensures that $\mathsf h (W_2) = \mathsf h (2W_2) \ge 3$.
If $\mathsf v_{\overline{y}} (W_2) \ge 2$, then
\[
  \mathsf v_g \big( \phi_H (2W') \big) \, = \, \mathsf v_g \Big( \phi_H \big(2(W_1 \bdot (\overline{x} \bdot \overline{x})^{[-1]} \big) \Big) + \mathsf v_g \Big( \phi_H \big(2(W_2 \bdot (\overline{y} \bdot \overline{y})^{[-1]} \big) \Big) \, \le \, \ell + 3|H| - 3 \,,
\]
a contradiction. Suppose that $\mathsf v_{\overline{y}} (W_2) \le 1$. Then we infer that $\mathsf v_{\overline{y}} (W \bdot (\overline{x} \bdot \overline{x})^{[-1]}) = \mathsf h (W \bdot (\overline{x} \bdot \overline{x})^{[-1]})$. It follows by $\mathsf h (W_2) \ge 3$ that there exists $\overline{z} \in \supp(W_2)$ with $\mathsf v_{\overline{z}} (W_2) = \mathsf h(W_2) \ge 3$. Then we assert that $\mathsf v_{\overline{x}} (W) = \mathsf h (W) \le \ell + 2$. Assume to the contrary that $\mathsf v_{\overline{x}} (W) = \mathsf h (W) \ge \ell + 3$.
Since $\mathsf v_{\overline{x}} (W_1) \ge 2$, {\bf A2} implies $W_1 = \overline{x}^{[\ell + 2]}$ with $\mathsf v_{\overline{x}} (W_2) = 1$, whence $\overline{y} = \overline{x}$. By the main case hypothesis, we have $\mathsf v_{2\overline{x}} (2W') = \mathsf v_{\overline{x}} (W') = \ell - 1$. Since $g \in (2(\Z/2n\Z))/H$ is the only element satisfying $\mathsf v_g \big( \phi_H (2W') \big) \ge \ell + 1 \ge 3$, it follows again by the main case hypothesis that $g = 2 \overline{z}$, $W_2 = \overline{x} \bdot \overline{z}^{[\ell+1]}$, and $\mathsf v_{\overline{z}} (W \bdot (\overline{x} \bdot \overline{x})^{[-1]}) = \mathsf h (W \bdot (\overline{x} \bdot \overline{x})^{[-1]})$. By swapping the role between $\overline{y}$ and $\overline{z}$, the argument used in the case above when $\mathsf v_{\overline{y}} (W_2) \ge 2$ leads to a contradiction. Thus $\mathsf v_{\overline{x}} (W) = \mathsf h (W) \le \ell + 2$, and then the swapping argument again leads to a contradiction.
Since $g \in (2(\Z/2n\Z))/H$, by shifting if necessary, we can assume that $g = H$, whence $|(2W_i)_H| \ge 3|H| - 2$ for all $i \in [1,2]$. By the same lines of the proof of {\bf SUBCASE 1.1}, we get a contradiction to $S \in \mathcal A (G)$.

\medskip
\noindent
{\bf SUBCASE 2.2.} \, $H$ is trivial.
\smallskip

Since $\ell = \frac{|W'|}{2} \ge n - 1$, it follows that $\mathsf v_g (2W') \, = \,\mathsf v_g \big( \phi_H (2W') \big) \, \ge \, |W'| + 2 - n \, \ge \, \ell + 1$, and by {\bf A2},
\[
  \mathsf h (2W) \, = \, \mathsf v_{2\overline{x}} (2W) \, = \,\mathsf v_{2\overline{x}} (2W_1) + \mathsf v_{2\overline{x}} (2W_2) \, \le \, (\ell + 2) + 1 \, = \, \ell + 3 \,.
\]
Thus we have $\mathsf v_{2\overline{x}} (2W') \le \mathsf v_{2\overline{x}} (2W) - 2 \le \ell + 1$.

Suppose that $\ell = 1$. Then $|W| = 6$, $n = 2$, and $|2W'| = 2$. Hence $\mathsf v_g (2W') = 2$ and $W = \overline{x}^{[2]} \bdot \overline{y}^{[2]} \bdot \overline{w_1} \bdot \overline{w_2}$ for some $\overline{w_1}, \overline{w_2} \in \Z/2n\Z$ with $2\overline{w_1} = 2\overline{w_2} = g$. If $\overline{w_1} = \overline{w_2} + \overline{n}$, then $\psi(\overline{w_1} \bdot \overline{w_2})$ and $\psi \big(\overline{x}^{[2]} \bdot \overline{y}^{[2]} \big)$ are both product-one sequences, a contradiction. Therefore $\overline{w_1} = \overline{w_2}$. Since $\ord (\alpha^{i}\tau) = 4$ for all $i \in [0, 2n-1]$ and $\psi(W)$ is a product-one sequence, we obtain that $| \{ \overline{x}, \overline{y}, \overline{w_1} \} | \ge 2$. Since $\mathsf v_{\overline{x}} (W) = \mathsf h (W) \ge 3$ and $\mathsf h \big( W \bdot (\overline{x} \bdot \overline{x})^{[-1]} \big) \ge 2$, it follows that either $\overline{x} = \overline{y}$ or $\overline{x} = \overline{w_1}$. Since $\sigma(W_1) = \sigma(W_2) + |W_1| \overline{n}$, we have
\[
  W \, = \, W_1 \bdot W_2 \, = \, \overline{x}^{[3]} \bdot \big( \overline{x} \bdot \overline{w}^{[2]} \big)
\]
for some $\overline{w} \in \Z/4\Z$ with $2\overline{w} \neq 2\overline{x}$. Thus $\psi(W)$ is the desired sequence for (b).

Suppose that $\ell \ge 2$. We assume to the contrary that $\mathsf v_{\overline{y}} (W_2) \ge 3$ and $\mathsf h (W) \le \ell + 2$. Since $\mathsf v_{2\overline{x}} (2W) = \mathsf v_{\overline{x}} (W) \le \ell + 2$, it follows that $\mathsf v_{2\overline{x}} (2W') \le \ell$, whence $g \neq 2 \overline{x}$. In view of $\mathsf v_{\overline{x}} (W_1) \ge 2$, $\mathsf v_{\overline{y}} (W_2) \ge 2$, and {\bf A2}, we must have $2 \overline{y} \neq 2 \overline{x}$. Let $g = 2 \overline{z}$ for some $\overline{z} \in \Z/2n\Z$.
If $g \neq 2 \overline{y}$, then by the main case hypothesis, $\overline{x}$, $\overline{y}$ and $\overline{z}$ are all distinct elements with $\mathsf v_{\overline{x}} (W) \ge \mathsf v_{\overline{z}} (W) \ge \ell +1$ and $\mathsf v_{\overline{y}} (W) \ge 3$, implying $2\ell + 4 = |W| \ge 2(\ell + 1) + 3 = 2\ell + 5$, a contradiction. Thus $g = 2 \overline{y}$, and again by the main case hypothesis, we have $\overline{z} = \overline{y}$. Hence $\mathsf v_{\overline{y}} \big( W \bdot (\overline{x} \bdot \overline{x})^{[-1]} \big) = \mathsf v_{\overline{z}} (W') + 2 \ge \ell + 3$, contradicting that $\mathsf h (W) \le \ell + 2$.

Therefore $\mathsf v_{\overline{y}} \big( W \bdot (\overline{x} \bdot \overline{x})^{[-1]} \big) = \mathsf h \big( W \bdot (\overline{x} \bdot \overline{x})^{[-1]} \big)$, and in view of the main case hypothesis, we have
\[
  3 \, \le \, \ell + 1 \, \le \, \mathsf v_g (2W') \, \le \, \mathsf v_{2\overline{y}} \big( 2 \big( W \bdot (\overline{x} \bdot \overline{x})^{[-1]} \big) \big) \, \le \, \mathsf v_{2\overline{x}} (2W) \,.
\]
Then it follows by $|2W| = 2 \ell + 4$ and $\mathsf h (2W) \le \ell + 3$ that $| \{ 2\overline{x}, 2\overline{y}, g \} | = 2$.
If $2\overline{y} = g$, then $2\overline{x} \neq 2\overline{y}$ and $\mathsf v_{2\overline{x}} (2W) \ge \mathsf v_{2\overline{y}} (2W) \ge \ell + 3$, whence $2 \ell + 4 = |2W| \ge 2 \ell + 6$, a contradiction. Thus $2\overline{y} \neq g$.

If $2\overline{x} = 2\overline{y}$, then $\mathsf v_{2\overline{x}} (2W) = 2 + \mathsf v_{2\overline{y}} \big( 2 \big( W \bdot (\overline{x} \bdot \overline{x})^{[-1]} \big)\big) \ge \ell + 3$ implies that $\mathsf v_{2\overline{x}} (2W) = \ell + 3$ and $\mathsf v_g (2W') = \ell + 1$. If $|W| \ge 2n + 4$, then $\ell \ge n$, and hence $\ell + 1 = \mathsf v_g (2W') \ge |W'| + 2 - n \ge \ell + 2$, a contradiction. Thus $|W| = 2n + 2$ and $\ell = n - 1$. Since $\mathsf v_{\overline{x}} (W_1) \ge \mathsf v_{\overline{x}} (W_2)$, we have $\mathsf v_{2\overline{x}} (2W_1) \ge \mathsf v_{2\overline{x}} (2W_2)$, and hence {\bf A2} ensures that $\mathsf v_{2\overline{x}} (2W_2) = 1$. It follows in view of the main case hypothesis that
\[
  W \, = \, W_1 \bdot W_2 \, = \, \overline{x}^{[n+1]} \bdot \big( \overline{x} \bdot \overline{z}^{[n]} \big) \,,
\]
where $\overline{z} \in \Z/2n\Z$ with $2\overline{z} = g \neq 2\overline{x}$. Since $\sigma(W_1) = \sigma(W_2) + |W_1| \overline{n}$, we have $nz + x \equiv (n+1)(x+n) \pmod{2n}$. Therefore $\psi(W)$ is the desired sequence for (b).

If $2\overline{x} = g$, then $\mathsf v_{2\overline{x}} (2W) \ge 2 + \mathsf v_g (2W') \ge \ell + 3$ implies that $\mathsf v_{2\overline{x}} (2W) = \ell + 3$ and $\mathsf v_{2\overline{y}} (2W) = \ell + 1$. The same argument as shown above ensures that $W = W_1 \bdot W_2 = \overline{x}^{[n+1]} \bdot \big( \overline{x} \bdot \overline{y}^{[n]} \big)$, where $\overline{x}, \overline{y} \in \Z/2n\Z$ with $2\overline{x} \neq 2\overline{y}$, and $ny + x \equiv (n+1)(x+n) \pmod{2n}$. Thus $\psi(W)$ is the desired sequence for (b).

\smallskip
To prove the ``In particular'' statement, we assume to the contrary that there exists a minimal product-one sequence $S$ such that $S = S_1 \bdot S_2$, where $S_1 \in \mathcal F \big( \la \alpha \ra \big)$ and $S_2 \in \mathcal F \big( G \setminus \la \alpha \ra \big)$ of length $|S_2| \ge 2n + 4$. Then we suppose that $S_2 = \prod^{\bullet}_{i \in [1,|S_2|]} \alpha^{x_i}\tau$ and $S_1 = T_1 \bdot T_2$ such that $\pi^{*} (T_1) (\alpha^{x_1}\tau) \pi^{*} (T_2) (\alpha^{x_2}\tau \ldots \alpha^{x_{|S_2|}}\tau) = 1_G$. Since $S \in \mathcal A (G)$, it follows that
\[
  S'' \, = \, \big( \pi^{*} (T_1)\alpha^{x_1}\tau \big) \bdot \big( \pi^{*} (T_2)\alpha^{x_2}\tau \big) \bdot \Big( {\small \prod}^{\bullet}_{i \in [3,|S_2|]} \alpha^{x_i}\tau \Big) \, \in \, \mathcal A \big( G \setminus \la \alpha \ra \big)
\]
and $|S''| = |S_2| \ge 2n + 4$, a contradiction to the main statement.
\end{proof}


\bigskip
\section{The main results} \label{4}
\bigskip

\smallskip
\begin{theorem} \label{4.1}~
Let $G$ be a dihedral group of order $2n$, where $n \in \N_{\ge 3}$ is odd.
A sequence $S$ over $G$ of length $\mathsf D (G)$ is a minimal product-one sequence if and only if it has one of the following two forms{\rm \,:}
\begin{enumerate}
\item[(a)] There exist $\alpha, \, \tau \in G$ such that $G = \la \alpha, \tau \t \alpha^{n} = \tau^{2} = 1_G \und \tau\alpha = \alpha^{-1}\tau \ra$ and $S = \alpha^{[2n-2]} \bdot \tau^{[2]}$.

\smallskip
\item[(b)] There exist $\alpha, \, \tau \in G$ and $i, j \in [0, n-1]$ with $\gcd (i-j, n) = 1$ such that $G = \la \alpha, \tau \t \alpha^{n} = \tau^{2} = 1_G \und \tau\alpha = \alpha^{-1}\tau \ra$ and $S = (\alpha^{i}\tau)^{[n]} \bdot (\alpha^{j}\tau)^{[n]}$.
\end{enumerate}
\end{theorem}

\begin{proof}
We fix $\alpha, \tau \in G$ such that $G = \la \alpha, \tau \t \alpha^{n} = \tau^{2} = 1_G \und \tau\alpha = \alpha^{-1}\tau \ra$. Then
\[
  G \, = \, \big\{ \alpha^{i} \t i \in [0, n-1] \big\} \, \cup \, \big\{ \alpha^{i}\tau \t i \in [0, n-1] \big\} \,.
\]
Let $G_0 = G \setminus \la \alpha \ra$. If $|S_{G_0}| = 0$, then $S \in \mathcal F \big( \la \alpha \ra \big)$, and since $|S| = 2n > \mathsf D \big( \la \alpha \ra \big) = n$, it follows that $S$ is not a minimal product-one sequence, a contradiction. Since $S$ is a product-one sequence, we have that $|S_{G_0}|$ is even. We distinguish three cases depending on $|S_{G_0}|$.

\medskip
\noindent
{\bf CASE 1.} \, $|S_{G_0}| \, = \, 2$.
\smallskip

Then we may assume by changing generating set if necessary that $S = T_1 \bdot \tau \bdot T_2 \bdot (\alpha^{x} \tau)$ with $\pi^{*} (T_1) (\tau) \pi^{*} (T_2) (\alpha^{x}\tau) = 1_G$, where $x \in [0, n-1]$ and $T_1, T_2 \in \mathcal F \big( \la \alpha \ra \big)$. Since $S \in \mathcal A (G)$, it follows that $T_1$ and $T_2$ must be both product-one free sequences, and thus $|T_1| = |T_2| = n-1$. Then we may assume by Lemma \ref{2.2}.1 that
\[
  T_1 \, = \, \alpha^{[n-1]} \quad \und \quad T_2 \, = \, (\alpha^{j})^{[n-1]} \,,
\]
where $j \in [0,n-1]$ with $\gcd (j,n) = 1$. Since $\pi^{*} (T_1) (\tau) \pi^{*} (T_2) (\alpha^{x}\tau) = 1_G$, it follows that $-1 \equiv -j + x \pmod{n}$, and thus it suffices to show that $x = 0$; Indeed, if this holds, then $j = 1$, whence $S = \alpha^{[2n-2]} \bdot \tau^{[2]}$ which is the desired sequence for (a).

Assume to the contrary that $x \in [1,n-1]$ so that $j \neq 1$.

\medskip
\noindent
{\bf SUBCASE 1.1.} \, $j$ is even.
\smallskip

Let $S_1 = \alpha^{j} \bdot \alpha^{[n-j]} \in \mathcal B (G)$. Since $j$ is even and $n$ is odd, $-1 \equiv -j + x \pmod n$ implies that
\[
  S_2 \, = \, \alpha^{[\frac{j-2}{2}]} \bdot (\alpha^{j})^{[\frac{n-3}{2}]} \bdot (\alpha^{j} \bdot \tau) \bdot \alpha^{[\frac{j-2}{2}]} \bdot (\alpha^{j})^{[\frac{n-3}{2}]} \bdot (\alpha \bdot \alpha^{x}\tau) \, \in \, \mathcal B (G) \,,
\]
whence $S = S_1 \bdot S_2$ contradicts that $S \in \mathcal A (G)$.

\medskip
\noindent
{\bf SUBCASE 1.2.} \, $j$ is odd.
\smallskip

Since $-1 \equiv -j + x \pmod{n}$, we obtain that $x = j - 1$, whence $x$ is even. Then $n-1-x$ is even, and we obtain that
\[
  \Big(\alpha^{[\frac{n-1-x}{2}]} (\alpha^{j})^{[\frac{n-1}{2}]} \alpha^{[x]}\Big) \tau \Big(\alpha^{[\frac{n-1-x}{2}]} (\alpha^{j})^{[\frac{n-1}{2}]}\Big) \alpha^{x} \tau  \, = \, 1_G \,.
\]
Let $S_1 = \alpha^{[\frac{n-1-x}{2}]} \bdot (\alpha^{j})^{[\frac{n-1}{2}]} \bdot \alpha^{[x]} \in \mathcal F \big( \la \alpha \ra \big)$. Since $x$ is even, it follows that $|S_1| = n-1 + \frac{x}{2} \ge n$, and hence $S_1$ has a product-one subsequence $W$. Thus $W$ and $S \bdot W^{[-1]}$ are both product-one sequences, contradicting that $S \in \mathcal A (G)$.

\medskip
\noindent
{\bf CASE 2.} \, $|S_{G_0}| \, \in \, [4, 2n - 2]$.
\smallskip

Then we may assume by changing generating set if necessary that $S = T_1 \bdot \tau \bdot T_2 \bdot T_3 \bdot (\alpha^{x}\tau)$, where $x \in [0,n-1]$, $T_1, T_2 \in \mathcal F \big( \la \alpha \ra \big)$, and $T_3 \in \mathcal F (G_0)$ with $|T_3| = |S_{G_0}| - 2$. Moreover, we suppose that $\pi^{*} (T_1) (\tau) \pi^{*} \big(T_2 \bdot T'_3 \big) (\alpha^{x}\tau) = 1_G$, where $T_3 = \prod^{\bullet}_{i \in  [1,|T_3|]} g_i$ is an ordered sequence and $T'_3 = \prod^{\bullet}_{i \in [1,|T_3|/2]} (g_{2i-1} g_{2i}) \in \mathcal F \big( \la \alpha \ra \big)$.
Then $T_1$ and $T_2 \bdot T'_3$ are both product-one free sequences and
\[
  |T_1 \bdot T_2 \bdot T'_3| \, = \, \Big(2n - |S_{G_0}|\Big) + \frac{|S_{G_0}|-2}{2} \, \ge \, n \,.
\]
Let $T_1 = p_1 \bdot \ldots \bdot p_{|T_1|}$, $T_2 = f_1 \bdot \ldots \bdot f_{|T_2|}$, and $T'_3 = q_1 \bdot \ldots \bdot q_{|T'_3|}$. Then we consider
\begin{itemize}
\item $H_1 = \{ p_1, \, p_1p_2, \ldots , \, (p_1 \ldots p_{|T_1|}) \}$, and

\smallskip
\item $H_2 = \{ q_1, \, q_1q_2, \, \ldots, \, (q_1 \ldots q_{|T'_3|}), \, (q_1 \ldots q_{|T'_3|}f_1), \, (q_1 \ldots q_{|T'_3|}f_1f_2), \, \ldots, \, (q_1 \ldots q_{|T'_3|}f_1 \ldots f_{|T_2|}) \}$.
\end{itemize}
Since both $T_1$ and $T_2 \bdot T'_3$ are product-one free, it follows that $H_1, H_2 \subset \la \alpha \ra \setminus \{ 1_G \}$ with $|H_1| = |T_1|$, $|H_2| = |T_2 \bdot T'_3|$, and $|H_1| + |H_2| = |T_1 \bdot T_2 \bdot T'_3| \ge n$.
Since $|\la \alpha \ra| = n$, we obtain that $H_1 \cap H_2 \neq \emptyset$, and hence we infer that there exist $W_1 \t T_1$, $W_2 \t T_2$, and $W'_3 \t T'_3$ such that $W'_3$ is a non-trivial sequence and $\pi^{*} (W_1) = \pi^{*} (W_2 \bdot W'_3)$. Let $W_3$ denote the corresponding subsequence of $T_3$ and assume that $W_3 = (\alpha^{y_1}\tau) \bdot (\alpha^{y_2}\tau) \bdot W''_3$. Then $Z = W_2 \bdot (\alpha^{y_1}\tau) \bdot W_1 \bdot (\alpha^{y_2}\tau) \bdot W''_3$ and $S \bdot Z^{[-1]}$ are both product-one sequences, contradicting that $S \in \mathcal A (G)$.

\medskip
\noindent
{\bf CASE 3.} \, $|S_{G_0}| \, = \, 2n$.
\smallskip

Since $|S| = 2n = |S_{G_0}|$,  we may assume that
\[
  S \, = \, \alpha^{k_1}\tau \bdot \alpha^{\ell_1}\tau \bdot \ldots \bdot \alpha^{k_n}\tau \bdot \alpha^{\ell_n}\tau \quad \mbox{ with } \,\, \alpha^{k_1}\tau \alpha^{\ell_1}\tau \ldots \alpha^{k_n}\tau \alpha^{\ell_n}\tau \, = \, 1_G \,,
\]
where $k_1, \ldots, k_n, \ell_1, \ldots, \ell_n \in [0, n-1]$. Then we set $S' = a^{k_1 - \ell_1} \bdot \ldots \bdot a^{k_n - \ell_n} \in \mathcal B \big( \la \alpha \ra \big)$ of length $|S'| = n$. Since $S \in \mathcal A (G)$, it follows that $S' \in \mathcal A \big( \la \alpha \ra \big)$, and by applying Lemma \ref{2.2}.1,
\begin{equation} \label{equal}~
k_1 - \ell_1 \, \equiv \, k_2 - \ell_2 \, \equiv \, \ldots \, \equiv \, k_n - \ell_n \, \pmod{n}
\end{equation}
with $\gcd (k_i - \ell_i, n) = 1$ for all $i \in [1,n]$. Let $j \in [1, n-1]$. Then we observe that
\[
  \alpha^{k_j}\tau  \alpha^{\ell_j}\tau  \alpha^{k_{j+1}}\tau \, = \, \alpha^{k_j - \ell_j + k_{j+1}}\tau \, = \, \alpha^{k_{j+1}}\tau  \alpha^{\ell_j}\tau  \alpha^{k_j}\tau \,.
\]
By swapping the role between $\alpha^{k_j}\tau$ and $\alpha^{k_{j+1}}\tau$, we obtain that
\[
  S'' \, = \, \alpha^{k_1 - \ell_1} \bdot \ldots \bdot \alpha^{k_{j+1} - \ell_j} \bdot \alpha^{k_{j} - \ell_{j+1}} \bdot \ldots \bdot \alpha^{k_n - \ell_n} \, \in \, \mathcal A \big( \la \alpha \ra \big)
\]
of length $|S''| = n$. Hence it follows again by applying Lemma \ref{2.2}.1 that
\[
  k_1 - \ell_1 \, \equiv \, \ldots \, \equiv \, k_{j+1} - \ell_j \, \equiv \, k_{j} - \ell_{j+1} \, \equiv \, \ldots \, \equiv \, k_n - \ell_n \, \pmod{n} \,,
\]
and thus (\ref{equal}) ensures that $k_j = k_{j+1}$, whence $k_1 = k_2 = \ldots = k_n$. Similarly we also obtain that $\ell_1 = \ell_2 = \ldots = \ell_n$, whence $S = (\alpha^{k_1}\tau)^{[n]} \bdot (\alpha^{\ell_1}\tau)^{[n]}$ with $\gcd (k_1 - \ell_1, n) = 1$, which is the desired sequence for (b).
\end{proof}

\smallskip
\begin{theorem} \label{4.2}~
Let $G$ be a dihedral group of order $2n$, where $n \in \N_{\ge 4}$ is even.
A sequence $S$ over $G$ of length $\mathsf D (G)$ is a minimal product-one sequence if and only if there exist $\alpha, \, \tau \in G$ such that $G = \la \alpha, \tau \t \alpha^{n} = \tau^{2} = 1_G \und \tau\alpha = \alpha^{-1}\tau \ra$ and $S \, = \, \alpha^{[n + \frac{n}{2} -2]} \bdot \tau \bdot (\alpha^{\frac{n}{2}} \tau)$.
\end{theorem}

\begin{proof}
We fix $\alpha, \tau \in G$ such that $G = \la \alpha, \tau \t \alpha^n = \tau^2 = 1_G \und \tau \alpha = \alpha^{-1} \tau \ra$. Then
\[
  G \, = \, \big\{ \alpha^i \t i \in [0,n-1] \big\} \, \cup \, \big\{ \alpha^i \tau \t i \in [0,n-1] \big\} \,.
\]
Let $G_0 = G \setminus \la \alpha \ra$. If $|S_{G_0}| = 0$, then $S \in \mathcal F \big( \la \alpha \ra \big)$, and since $|S| = n + \frac{n}{2} > \mathsf D \big( \la \alpha\ra \big) = n$, it follows that $S$ is not a minimal product-one sequence, a contradiction. Since $S$ is a product-one sequence, Proposition \ref{3.2} ensures that $|S_{G_0}| \in [2,n]$ is even. We distinguish two cases depending on $|S_{G_0}|$.

\medskip
\noindent
{\bf CASE 1.} \, $|S_{G_0}| \, = \, 2$.
\smallskip

Then we may assume by changing generating set if necessary that $S = T_1 \bdot \tau \bdot T_2 \bdot (\alpha^{x} \tau)$ with $\pi^{*} (T_1)  (\tau) \pi^{*} (T_2)  (\alpha^{x}\tau) = 1_G$, where $x \in [0,n-1]$ and $T_1, T_2 \in \mathcal F \big( \la \alpha \ra \big)$. Since $S \in \mathcal A (G)$, it follows that $T_1$ and $T_2$ must be both product-one free sequences.

If $|T_1| \ge \frac{n}{2}$ and $|T_2| \ge \frac{n}{2}$, then $T^{2}_1$ and $T^{2}_2 \in \mathcal F \big( \la \alpha^{2} \ra \big)$ (see (\ref{double})) with $|T^{2}_1| \ge \frac{n}{2}$ and $|T^{2}_2| \ge \frac{n}{2}$, and it follows by $\mathsf D \big( \la \alpha^{2} \ra \big) = \frac{n}{2}$ that there exist $W_1 \t T_1$ and $W_2 \t T_2$ such that $W^{2}_1$ and $W^{2}_2$ are product-one sequences over $\la \alpha^{2} \ra$. Since $T_1$ and $T_2$ are product-one free, we obtain that $\pi^{*} (W_1) = \alpha^{\frac{n}{2}} = \pi^{*} (W_2)$. Therefore $W_1 \bdot W_2$ and $S \bdot (W_1 \bdot W_2)^{[-1]}$ are both product-one sequences, contradicting that $S \in \mathcal A (G)$.

Thus either $|T_1| \le \frac{n}{2} -1$ or $|T_2| \le \frac{n}{2} -1$, and we may assume that $|T_1| = \frac{n}{2} -1$ and $|T_2| = n -1$. Then Lemma \ref{2.2}.1 implies that $T_2 = (\alpha^{j})^{[n-1]}$ for some odd $j \in [1,n-1]$. Then we may assume by changing generating set if necessary that $j = 1$ so that $S = T_3 \bdot \tau \bdot \alpha^{[n-1]} \bdot (\alpha^{y}\tau)$, where $y \in [0,n-1]$ and $T_3 \in \mathcal F \big( \la \alpha \ra \big)$. Since $T_3 \bdot \alpha \bdot \tau \bdot (\alpha^{y}\tau)$ is a product-one  sequence, we have that
\[
  T_3 \bdot \alpha^{[\frac{n}{2}]} \bdot \tau \bdot \alpha^{[\frac{n}{2}-1]} \bdot (\alpha^{y}\tau) \, \in \, \mathcal B (G) \,.
\]
It follows that $T_3 \bdot \alpha^{[\frac{n}{2}]}$ is a product-one free sequence of length $n-1$, and again by Lemma \ref{2.2}.1 that $T_3 = \alpha^{[\frac{n}{2} -1]}$. Since $(\frac{n}{2} -1) \equiv (n-1) + y \pmod{n}$, we infer that $y = \frac{n}{2}$, and the assertion follows.

\medskip
\noindent
{\bf CASE 2.} \, $|S_{G_0}| \, \in \, [4, n]$.

\medskip
\noindent
{\bf SUBCASE 2.1.} \, $n = 4$.
\smallskip

Then we may assume by changing generating set if necessary that $S = \alpha^{r_1} \bdot \alpha^{r_2} \bdot \tau \bdot \alpha^{x}\tau \bdot \alpha^{y}\tau \bdot \alpha^{z}\tau$ for some $r_1, r_2 \in [1,3]$ and $x, y, z \in [0,3]$.
If $\alpha^{r_1} \alpha^{r_2}  \tau  \alpha^{x}\tau  \alpha^{y}\tau \alpha^{z}\tau = 1_G$, then $S' = \alpha^{r_1} \bdot \alpha^{r_2} \bdot \alpha^{-x} \bdot \alpha^{y-z} \in \mathcal A \big( \la \alpha \ra \big)$, and hence it follows by applying Lemma \ref{2.2}.1 that $r_1 \equiv r_2 \equiv -x \equiv y-z \equiv j \pmod{4}$ for some odd $j \in [1,3]$. Thus $S = S_1 \bdot S_2$, where $S_1 = \tau \bdot \alpha^{r_1} \bdot \alpha^{x}\tau$ and $S_2 = \alpha^{y}\tau \bdot \alpha^{r_2} \bdot \alpha^{z}\tau$ are both product-one sequences, contradicting that $S \in \mathcal A (G)$. Thus we can assume that $\alpha^{r_1}  \tau \alpha^{r_2} \alpha^{x}\tau  \alpha^{y}\tau \alpha^{z}\tau = 1_G$, and we consider
\[
  S'' \, = \, \alpha^{r_1} \bdot \alpha^{-r_2} \bdot \tau \bdot \alpha^{x}\tau \bdot \alpha^{y}\tau \bdot \alpha^{z}\tau \, \in \, \mathcal B (G) \,.
\]
Then, by the same argument as shown above, we obtain that $S'' \notin \mathcal A (G)$. Let $S'' = U_1 \bdot U_2$ for some $U_1, U_2 \in \mathcal B (G)$.
Since $S$ is a minimal product-one sequence, we must have that
\[
  U_1 \, = \, \alpha^{r_1} \bdot \alpha^{-r_2} \quad \und \quad U_2 \, = \, \tau \bdot \alpha^{x}\tau \bdot \alpha^{y}\tau \bdot \alpha^{z}\tau
\]
are both minimal product-one sequences, whence we obtain that $r_1 = r_2$.
Since $U_2 \in \mathcal A (G)$, Proposition \ref{3.2} implies that
\[
  U_2 \, = \, \tau \bdot \alpha\tau \bdot \alpha^{2}\tau \bdot \alpha^{3}\tau \quad \mbox{ or } \quad U_2 \, = \, (\alpha^{x_1}\tau)^{[2]} \bdot \alpha^{y_1}\tau \bdot \alpha^{y_1+2}\tau \,,
\]
where $x_1, y_1 \in [0,3]$ with $x_1 \equiv y_1 + 1 \pmod{2}$. Since $S \in \mathcal A (G)$, we obtain that either $r_1 = r_2 = 1$ or $r_1 = r_2 = 3$. If $r_1 = r_2 = 1$, then
\[
  S \, = \, (\alpha \bdot \tau \bdot \alpha\tau) \bdot (\alpha \bdot \alpha^{2}\tau \bdot \alpha^{3}\tau) \quad \mbox{ or } \quad S \, = \, (\alpha \bdot \alpha^{x_1}\tau \bdot \alpha^{y_1}\tau) \bdot (\alpha^{x_1}\tau \bdot \alpha \bdot \alpha^{y_1 +2}\tau) \,,
\]
contradicting that $S \in \mathcal A (G)$. If $r_1 = r_3 = 3$, then
\[
  S \, = \, (\tau \bdot \alpha^{3} \bdot \alpha\tau) \bdot (\alpha^{2}\tau \bdot \alpha^{3} \bdot \alpha^{3}\tau) \quad \mbox{ or } \quad S \, = \, (\alpha^{3} \bdot \alpha^{x_1}\tau \bdot \alpha^{y_1}\tau) \bdot (\alpha^{x_1}\tau \bdot \alpha^{3} \bdot \alpha^{y_1 +2}\tau) \,,
\]
contradicting that $S \in \mathcal A (G)$.

\medskip
\noindent
{\bf SUBCASE 2.2.} \, $n \ge 6$.
\smallskip

Then we may assume by changing generating set if necessary that $S = T_1 \bdot \tau \bdot T_2 \bdot T_3 \bdot (\alpha^{x}\tau)$, where $x \in [0,n-1]$, $T_1, T_2 \in \mathcal F \big( \la \alpha \ra \big)$ with $|T_2| \ge |T_1| \ge 0$, and $T_3 \in \mathcal F (G_0)$ with $|T_3| = |S_{G_0}| - 2$. Moreover, we suppose that $\pi^{*} (T_1) (\tau) \pi^{*} \big(T_2 \bdot T'_3\big) (\alpha^{x}\tau) = 1_G$, where $T_3 = \prod^{\bullet}_{i \in [1,|T_3|]} g_i$ is an ordered sequence and $T'_3 = \prod^{\bullet}_{i \in [1, |T_3|/2]} (g_{2i-1}g_{2i}) \in \mathcal F \big( \la \alpha \ra \big)$.
Then $T_1$ and $T_2 \bdot T'_3$ are both product-one free sequences and
\[
  |T_1 \bdot T_2 \bdot T_3'| \, = \, \Big( n + \frac{n}{2} -|S_{G_0}| \Big) + \frac{|S_{G_0}| - 2}{2} \, \ge \, n - 1 \,.
\]

If $|T_1 \bdot T_2 \bdot T'_3| \ge n$, then we infer that there exist subsequences $W_1 \t T_1$, $W_2 \t T_2$, and $W'_3 \t T'_3$ such that $W'_3$ is a non-trivial sequence (this follows by the same argument as used in {\bf CASE 2} of Theorem \ref{4.1}) and $\pi^{*} (W_1) = \pi^{*} \big(W_2 \bdot W'_3\big)$. Let $W_3$ denote the corresponding subsequence of $T_3$ and assume that $W_3 = (\alpha^{y_1}\tau) \bdot (\alpha^{y_2}\tau) \bdot W''_3$. Then $Z = W_2 \bdot (\alpha^{y_1}\tau) \bdot W_1 \bdot (\alpha^{y_2}\tau) \bdot W_3''$ and $S \bdot Z^{[-1]}$ are both product-one sequences, contradicting that $S \in \mathcal A (G)$.

Suppose that $|T_1 \bdot T_2 \bdot T'_3| = n-1$. Then $|T'_3| = \frac{n}{2} - 1$ and $|T_2| \ge \frac{n}{4}$. Since $T_2 \bdot T'_3$ is a product-one free sequence with $|T_2 \bdot T'_3| \ge \frac{3n}{4} -1 \ge \frac{n+1}{2}$, it follows by Lemma \ref{2.2} that $T_2 \bdot T'_3$ is $g$-smooth for some $g \in \la \alpha \ra$ with $\ord (g) = n$, and for every $z \in \Pi(T_2 \bdot T'_3)$, there exists a subsequence $W \t T_2 \bdot T'_3$ with $g \t W$ such that $\pi^{*} (W) = z$. Since $|T'_3| = \frac{n}{2}-1$, Lemma \ref{2.2}.3 implies that $g \t T'_3$.

If $\Pi(T_1) \cap \Pi(T_2 \bdot T'_3) \neq \emptyset$, then there exist subsequences $W_1 \t T_1$, $W_2 \t T_2$, and $W'_3 \t T'_3$ such that $W'_3$ is a non-trivial sequence (this follows from the above paragraph that we can choose $W_2 \bdot W'_3 \t T_2 \bdot T'_3$ such that $g \t W_2 \bdot W'_3$ and $g \t T'_3$) and $\pi^{*} (W_1) = \pi^{*} (W_2 \bdot W'_3)$. Let $W_3$ denote the corresponding subsequence of $T_3$ and assume that $W_3 = (\alpha^{y_1}\tau) \bdot (\alpha^{y_2}\tau) \bdot W''_3$. Then $Z = W_2 \bdot (\alpha^{y_1}\tau) \bdot W_1 \bdot (\alpha^{y_2}\tau) \bdot W''_3$ and $S \bdot Z^{[-1]}$ are both product-one sequences, contradicting that $S \in \mathcal A (G)$.
Hence $\Pi(T_1) \cap \Pi(T_2 \bdot T'_3) = \emptyset$, and it follows that $T^{-1}_1 \bdot T_2 \bdot T'_3$ is a product-one free sequence of length $n-1$. By Lemma \ref{2.2}.1, there exists an odd $j \in [1,n-1]$ such that
\[
  T^{-1}_1 \bdot T_2 \bdot T'_3 \, = \, \big( \alpha^{j} \big)^{[n-1]} \,,
\]
and we may assume by changing generating set if necessary that $j = 1$ so that $x = 1$.
If $|T_1| \ge 1$, then
\[
  \big( \alpha \bdot \alpha^{-1} \big)^{[|T_1|]} \quad \und \quad \alpha^{[1+\frac{n-2-2|T_1|}{2}]} \bdot \tau \bdot \alpha^{[\frac{n-2-2|T_1|}{2}]} \bdot \alpha\tau
\]
are both product-one sequences, contradicting that $S \in \mathcal A (G)$. Thus $|T_1| = 0$, and then we obtain that $T_3 = \big( \alpha^{r + 1}\tau \bdot \alpha^{r}\tau \big)^{[\frac{n}{2} - 1]}$ for some $r \in [0,n-1]$ (this follows by the swapping argument as used in {\bf CASE 3} of Theorem \ref{4.1}). This implies that $S = (\alpha \bdot \tau \bdot \alpha\tau) \bdot (\alpha^{r + 1}\tau \bdot \alpha \bdot \alpha^{r}\tau)^{[\frac{n}{2} - 1]}$, contradicting that $S \in \mathcal A (G)$.
\end{proof}

\smallskip
\begin{theorem} \label{4.3}~
Let $G$ be a dicyclic group of order $4n$, where $n \ge 2$.
A sequence $S$ over $G$ of length $\mathsf D (G)$ is a minimal product-one sequence if and only if there exist $\alpha, \, \tau \in G$ such that $G = \la \alpha, \tau \t \alpha^{2n} = 1_G, \, \tau^{2} = \alpha^{n}, \und \tau\alpha = \alpha^{-1}\tau \ra$ and $S = \alpha^{[3n - 2]} \bdot \tau^{[2]}$.
\end{theorem}

\begin{proof}
We fix $\alpha, \tau \in G$ such that $G = \la \alpha, \tau \t \alpha^{2n} = 1_G, \tau^{2} = \alpha^{n}, \und \tau\alpha = \alpha^{-1}\tau \ra$. Then
\[
  G \, = \, \big\{ \alpha^{i} \t i\in [0,2n-1] \big\} \, \cup \, \big\{ \alpha^{i}\tau \t i\in [0,2n-1] \big\} \,.
\]
Let $G_0 = G \setminus \la \alpha \ra$. If $|S_{G_0}| = 0$, then $S \in \mathcal F \big( \la \alpha \ra \big)$, and since $|S| = 3n > \mathsf D \big( \la \alpha \ra \big) = 2n$, it follows that $S$ is not a minimal product-one sequence, a contradiction. Since $S$ is a product-one sequence, Proposition \ref{3.3} ensures that $|S_{G_0}| \in [2, 2n+2]$ is even. We distinguish two cases depending on $|S_{G_0}|$.

\medskip
\noindent
{\bf CASE 1.} \, $|S_{G_0}| \, = \, 2$.
\smallskip

Then we may assume by changing generating set if necessary that $S = T_1 \bdot \tau \bdot T_2 \bdot (\alpha^x\tau)$ with $\pi^{*} (T_1) (\tau) \pi^{*}(T_2) (\alpha^{x}\tau) = 1_G$, where $x \in [0, 2n-1]$ and $T_1, T_2 \in \mathcal F \big( \la \alpha \ra \big)$. Since $S \in \mathcal A (G)$, it follows that $T_1$ and $T_2$ must be both product-one free sequences.

If $|T_1| \ge n$ and $|T_2| \ge n$, then $T^{2}_1$ and $T^{2}_2 \in \mathcal F \big( \la \alpha^{2} \ra \big)$ (see (\ref{double})) with $|T^{2}_1| \ge n$ and $|T^{2}_2| \ge n$, and it follows by $\mathsf D \big( \la \alpha^{2} \ra \big) = n$ that there exist $W_1 \t T_1$ and $W_2 \t T_2$ such that $W^{2}_1$ and $W^{2}_2$ are product-one sequence over $\la \alpha^{2} \ra$. Since $T_1$ and $T_2$ are product-one free, we obtain that $\pi^{*} (W_1) = \alpha^{n} = \pi^{*} (W_2)$. Therefore $W_1 \bdot W_2$ and $S \bdot (W_1 \bdot W_2)^{[-1]}$ are both product-one sequences, contradicting that $S \in \mathcal A (G)$.

Thus either $|T_1| \le n-1$ or $|T_2| \le n-1$, and we may assume that $|T_1| = n-1$ and $|T_2| = 2n-1$. Then Lemma \ref{2.2}.1 implies that $T_2 = (\alpha^{j})^{[2n-1]}$ for some odd $j \in [1,2n-1]$. Then we may assume by changing generating set if necessary that $j = 1$ so that $S = T_3 \bdot \tau \bdot \alpha^{[2n-1]} \bdot (\alpha^{y}\tau)$, where $y \in [0,2n-1]$ and $T_3 \in \mathcal F \big( \la \alpha \ra \big)$. Since $T_3 \bdot \alpha \bdot \tau \bdot (\alpha^{y}\tau)$ is a product-one sequence, we have that
\[
  T_3 \bdot \alpha^{[n]} \bdot \tau \bdot \alpha^{[n-1]} \bdot (\alpha^{y}\tau) \, \in \, \mathcal B (G) \,.
\]
It follows that $T_3 \bdot \alpha^{[n]}$ is a product-one free sequence of length $2n-1$, and again by Lemma \ref{2.2}.1 that $T_3 = \alpha^{[n-1]}$. Since $(n-1) \equiv (2n-1) + y + n \pmod{2n}$, we infer that $y = 0$, and the assertion follows.

\medskip
\noindent
{\bf CASE 2.} \, $|S_{G_0}| \, \in \, [4, 2n + 2]$.

\medskip
\noindent
{\bf SUBCASE 2.1.} \, $n = 2$
\smallskip

Then $G = Q_8$ is the quaternion group. If $|S_{G_0}| = 6$, then by Proposition \ref{3.3}, we have that
\[
  S \, = \, (\alpha^{x}\tau)^{[4]} \, \bdot \, (\alpha^{y}\tau)^{[2]} \,
\]
where $x, y \in [0,3]$ such that $2x \not\equiv 2y \pmod{4}$ and $2y + x \equiv 3(x+2) \pmod{4}$. Since $2y \equiv 2x + 2 \pmod{4}$, it follows by letting $\alpha_1 = \alpha^{x}\tau$ and $\tau_1 = \alpha^{y}\tau$ that $S = \alpha^{[4]}_1 \bdot \tau^{[2]}_1$, where $G = \la \alpha_1, \tau_1 \t \alpha^{4}_1 = 1_G, \, \tau^{2}_1 = \alpha^{2}_1, \und \tau_1 \alpha_1 = \alpha^{-1}_1 \tau_1 \ra$, whence the assertion follows.

Suppose that $|S_{G_0}| = 4$, and we may assume by changing generating set if necessary that $S = \alpha^{r_1} \bdot \alpha^{r_2} \bdot \tau \bdot \alpha^{x}\tau \bdot \alpha^{y}\tau \bdot \alpha^{z}\tau$ for some $r_1, r_2 \in [1,3]$ and $x, y, z \in [0,3]$.
If $\alpha^{r_1} \alpha^{r_2}  \tau  \alpha^{x}\tau  \alpha^{y}\tau \alpha^{z}\tau = 1_G$, then $S' = \alpha^{r_1} \bdot \alpha^{r_2} \bdot \alpha^{-x+2} \bdot \alpha^{y-z+2} \in \mathcal A \big( \la \alpha \ra \big)$, and hence it follows by applying Lemma \ref{2.2}.1 that $r_1 \equiv r_2 \equiv -x+2 \equiv y-z+2 \equiv j \pmod{4}$ for some odd $j \in [1,3]$. Thus $S = S_1 \bdot S_2$, where $S_1 = \alpha^{r_1} \bdot \alpha^{x}\tau \bdot \tau$ and $S_2 = \alpha^{r_2} \bdot \alpha^{z}\tau \bdot \alpha^{y}\tau$ are both product-one sequences, contradicting that $S \in \mathcal A (G)$. Hence we can assume that $\alpha^{r_1} \tau \alpha^{r_2} \alpha^{x}\tau  \alpha^{y}\tau \alpha^{z}\tau = 1_G$, and we consider
\[
  S'' \, = \, \alpha^{r_1} \bdot \alpha^{-r_2} \bdot \tau \bdot \alpha^{x}\tau \bdot \alpha^{y}\tau \bdot \alpha^{z}\tau \, \in \, \mathcal B (G) \,.
\]
Then, by the same argument as shown above, we obtain that $S'' \notin \mathcal A (G)$. Let $S'' = U_1 \bdot U_2$ for some $U_1, U_2 \in \mathcal B (G)$.
Since $S$ is a minimal product-one sequence, we must have that
\[
  U_1 \, = \, \alpha^{r_1} \bdot \alpha^{-r_2} \quad \und \quad U_2 \, = \, \tau \bdot \alpha^{x}\tau \bdot \alpha^{y}\tau \bdot \alpha^{z}\tau
\]
are both minimal product-one sequences. Then $r_1 = r_2$, and we may assume that $\tau \alpha^{x}\tau \alpha^{y}\tau \alpha^{z}\tau = 1_G$. Then $U_2 \in \mathcal A (G)$ implies that $\alpha^{-x+2} \bdot \alpha^{y-z+2} \in \mathcal A \big( \la \alpha \ra \big)$, whence $x \equiv y-z \pmod{4}$. Since $(\alpha^{2}\tau \bdot \tau) \bdot (\alpha^{2}\tau \bdot \tau)$ is not a minimal product-one sequence, it follows by case distinction on $x, y, z$ that we have
\[
  \begin{aligned}
   U_2 \quad \in & \quad \big\{ \tau^{[4]}, \,\, \tau^{[2]} \bdot (\alpha\tau)^{[2]}, \,\, \tau^{[2]} \bdot (\alpha^{3}\tau)^{[2]}, \,\, \tau^{[2]} \bdot \alpha\tau \bdot \alpha^{3}\tau, \\
                 & \quad \quad \tau \bdot (\alpha\tau)^{[2]} \bdot \alpha^{2}\tau, \,\, \tau \bdot \alpha^{2}\tau \bdot (\alpha^{3}\tau)^{[2]}, \,\, \tau \bdot \alpha\tau \bdot \alpha^{2}\tau \bdot \alpha^{3}\tau \big\} \,.
  \end{aligned}
\]
Since $S \in \mathcal A (G)$, we can assume by changing the generator $\alpha$ for $\alpha^{3}$ if necessary that $r_1 = r_2 = 1$, and thus we must have $U_2 = \tau^{[4]}$, for otherwise, $S$ is the product of two product-one sequences, a contradiction. By letting $\alpha_1 = \tau$ and $\tau_1 = \alpha^{r_1}$, we obtain that $S = \alpha^{[4]}_1 \bdot \tau^{[2]}_1$, where $G = \la \alpha_1, \tau_1 \t \alpha^{4}_1 = 1_G, \, \tau^{2}_1 = \alpha^{2}_1, \und \tau_1 \alpha_1 = \alpha^{-1}_1 \tau_1 \ra$, whence the assertion follows.

\medskip
\noindent
{\bf SUBCASE 2.2.} \, $n \ge 3$.
\smallskip

Then we may assume by changing generating set if necessary that $S = T_1 \bdot \tau \bdot T_2 \bdot T_3 \bdot \alpha^{x}\tau$, where $x \in [0,2n-1]$, $T_1, T_2 \in \mathcal F \big( \la \alpha \ra \big)$ with $|T_2| \ge |T_1| \ge 0$, and $T_3 \in \mathcal F (G_0)$ with $|T_3| = |S_{G_0}| - 2$. Moreover, we suppose that $\pi^{*} (T_1)  (\tau) \pi^{*} \big(T_2 \bdot T'_3 \big) (\alpha^{x}\tau) = 1_G$, where $T_3 = \prod^{\bullet}_{i \in [1,|T_3|]} g_i$ is an ordered sequence and $T'_3 = \prod^{\bullet}_{i \in [1,|T_3|/2]} (g_{2i-1}g_{2i}) \in \mathcal F \big( \la \alpha \ra \big)$.
Then $T_1$ and $T_2 \bdot T'_3$ are both product-one free sequences and
\[
  |T_1 \bdot T_2 \bdot T'_3| \, = \, \Big( 3n - |S_{G_0}| \Big) + \frac{|S_{G_0}|-2}{2} \, \ge \, 2n - 2 \,.
\]

If $|T_1 \bdot T_2 \bdot T'_3| \ge 2n$, then we infer that there exists a product-one subsequence $Z$ of $S$ such that $S \bdot Z^{[-1]}$ is again a product-one sequence (this follows by the same line of the proof as used in {\bf SUBCASE 2.2} of Theorem \ref{4.2}), contradicting that $S \in \mathcal A (G)$.

Suppose that $|T_1 \bdot T_2 \bdot T'_3| = 2n-1$. Then $|T'_3| = n-1$ and $|T_2| \ge \frac{n}{2}$. Since $T_2 \bdot T'_3$ is a product-one free sequence with $|T_2 \bdot T'_3| \ge \frac{3n}{2}-1 \ge \frac{2n+1}{2}$, it follows by Lemma \ref{2.2} that $T_2 \bdot T'_3$ is $g$-smooth for some $g \in \la \alpha \ra$ with $\ord (g) = 2n$, and for every $z \in \Pi(T_2 \bdot T'_3)$, there exists a subsequence $W \t T_2 \bdot T'_3$ with $g \t W$ such that $\pi^{*} (W) = z$. Since $|T'_3| = n - 1$, Lemma \ref{2.2}.3 implies that $g \t T'_3$.

If $\Pi(T_1) \cap \Pi(T_2 \bdot T'_3) \neq \emptyset$, then there exist subsequences $W_1 \t T_1$, $W_2 \t T_2$, and $W'_3 \t T'_3$ such that $W'_3$ is a non-trivial sequence (as argued in similar cases) and $\pi^{*} (W_1) = \pi^{*} (W_2 \bdot W'_3)$. Let $W_3$ denote the corresponding subsequence of $T_3$ and assume that $W_3 = (\alpha^{y_1}\tau) \bdot (\alpha^{y_2}\tau) \bdot W''_3$. Then $Z = W_2 \bdot (\alpha^{y_1}\tau) \bdot W_1 \bdot (\alpha^{y_2}\tau) \bdot W''_3$ and $S \bdot Z^{[-1]}$ are both product-one sequences, contradicting that $S \in \mathcal A (G)$.
Hence $\Pi(T_1) \cap \Pi(T_2 \bdot T'_3) = \emptyset$, and it follows that $T^{-1}_1 \bdot T_2 \bdot T'_3$ is a product-one free sequence of length $2n-1$. By Lemma \ref{2.2}.1, there exists an odd $j \in [1,2n-1]$ such that
\[
  T^{-1}_1 \bdot T_2 \bdot T'_3 \, = \, \big( \alpha^{j} \big)^{[2n-1]} \,,
\]
and we may assume by changing generating set if necessary that $j = 1$ so that $x \equiv 1 + n \pmod{2n}$. Note that $2n - 2 - 2|T_1| \ge 0$ is even.
If $|T_1| \ge 1$, then
\[
  \big( \alpha \bdot \alpha^{-1} \big)^{[|T_1|]} \quad \und \quad \alpha^{[1+ \frac{2n-2-2|T_1|}{2}]} \bdot \tau \bdot \alpha^{[\frac{2n-2-2|T_1|}{2}]} \bdot (\alpha^{x}\tau)
\]
are both product-one sequences, contradicting that $S \in \mathcal A (G)$.
Thus $|T_1| = 0$, and we obtain that $T_3 = \big( \alpha^{r + 1}\tau \bdot \alpha^{r + n}\tau \big)^{[n-1]}$ for some $r \in [0,2n-1]$ (as argued in similar cases). Since $x \equiv 1 + n \pmod{2n}$, we obtain that $S = (\alpha \bdot \tau \bdot \alpha^{x}\tau) \bdot (\alpha \bdot \alpha^{r+n}\tau \bdot \alpha^{r+1}\tau)^{[n-1]}$, contradicting that $S \in \mathcal A (G)$.

Suppose that $|T_1 \bdot T_2 \bdot T'_3| = 2n-2$. Then $|T'_3| = n$ and $|T_2| \ge \frac{n}{2} - 1$. Since $T_2 \bdot T'_3$ is a product-one free sequence with $|T_2 \bdot T'_3| \ge \frac{3n}{2}-1 \ge \frac{2n+1}{2}$, it follows by Lemma \ref{2.2} that $T_2 \bdot T'_3$ is $g$-smooth for some $g \in \la \alpha \ra$ with $\ord (g) = 2n$, and for every $z \in \Pi(T_2 \bdot T'_3)$, there exists a subsequence $W \t T_2 \bdot T'_3$ with $g \t W$ such that $\pi^{*} (W) = z$. Since $|T'_3| \ge n - 1$, Lemma \ref{2.2}.3 implies that $g \t T'_3$.

If $\Pi(T_1) \cap \Pi(T_2 \bdot T'_3) \neq \emptyset$, then there exist subsequences $W_1 \t T_1$, $W_2 \t T_2$, and $W'_3 \t T'_3$ such that $W'_3$ is a non-trivial sequence (as argued in similar cases) and $\pi^{*} (W_1) = \pi^{*} (W_2 \bdot W'_3)$. Let $W_3$ denote the corresponding subsequence of $T_3$ and assume that $W_3 = (\alpha^{y_1}\tau) \bdot (\alpha^{y_2}\tau) \bdot W''_3$. Then $Z = W_2 \bdot (\alpha^{y_1}\tau) \bdot W_1 \bdot (\alpha^{y_2}\tau) \bdot W''_3$ and $S \bdot Z^{[-1]}$ are both product-one sequences, contradicting that $S \in \mathcal A (G)$.
Hence $\Pi(T_1) \cap \Pi(T_2 \bdot T'_3) = \emptyset$, and it follows that $T^{-1}_1 \bdot T_2 \bdot T'_3$ is a product-one free sequence of length $2n-2$. By Lemma \ref{2.2}.2, there exists an odd $j \in [1,2n-1]$ such that either
\[
  T^{-1}_1 \bdot T_2 \bdot T'_3 \, = \, \big( \alpha^{j} \big)^{[2n-3]} \bdot \alpha^{2j} \quad \mbox{ or } \quad T^{-1}_1 \bdot T_2 \bdot T'_3 \, = \, \big( \alpha^{j} \big)^{[2n-2]} \,,
\]
and we may assume by changing generating set if necessary that $j = 1$ so that either
\[
   T^{-1}_1 \bdot T_2 \bdot T'_3 \, = \, \alpha^{[2n-3]} \bdot \alpha^{2} \,, \quad \mbox{ whence } \,\, x \, \equiv 1 + n \pmod{2n} \,,
\]
or else
\[
  T^{-1}_1 \bdot T_2 \bdot T'_3 \, = \, \alpha^{[2n-2]} \,, \quad \mbox{ whence } \,\, x \, \equiv \, 2 + n \pmod{2n} \,.
\]

Suppose that $T^{-1}_1 \bdot T_2 \bdot T'_3 = \alpha^{[2n-3]} \bdot \alpha^{2}$ and $x \equiv 1 + n \pmod{2n}$.
If $|T_1| \ge 1$ and $\alpha^{-2} \in \supp(T_1)$, then
\[
  \big( \alpha^{-2} \bdot \alpha \bdot \alpha \big) \bdot \big( \alpha \bdot \alpha^{-1} \big)^{[|T_1|-1]} \quad \und \quad \alpha^{[1+\frac{2n-4-2|T_1|}{2}]} \bdot \tau \bdot \alpha^{[\frac{2n-4-2|T_1|}{2}]} \bdot \alpha^{x}\tau
\]
are both product-one sequences, contradicting that $S \in \mathcal A (G)$. If $|T_1| \ge 1$ and $\alpha^{-2} \notin \supp(T_1)$, then
\[
  \big( \alpha \bdot \alpha^{-1} \big)^{[|T_1|]} \quad \und \quad \alpha^{2} \bdot \alpha^{[\frac{2n-4-2|T_1|}{2}]} \bdot \tau \bdot \alpha^{[1+\frac{2n-4-2|T_1|}{2}]} \bdot \alpha^{x}\tau
\]
are both product-one sequences, contradicting that $S \in \mathcal A (G)$. Thus we obtain that $|T_1| = 0$.

If $\alpha^{2} \in \supp(T_2)$, then $T_3 = \big( \alpha^{r + 1}\tau \bdot \alpha^{r + n}\tau \big)^{[n]}$ for some $r \in [0,2n-1]$ (as argued in similar cases). Since $x \equiv 1 + n \pmod{2n}$, we obtain that
\[
  S_1 \, = \, \alpha^{r+1}\tau \bdot \alpha^{r+n}\tau \bdot \alpha^{x}\tau \bdot \alpha^{2} \bdot \tau \,, \quad S_2 \, = \, \big( \alpha^{r+1}\tau \bdot \alpha^{r+n}\tau \big)^{[2]} \,, \und S_3 \, = \, \alpha^{r+n}\tau \bdot \alpha^{r+1}\tau \bdot \alpha
\]
are all product-one sequences, whence $S = S_1 \bdot S_2 \bdot S^{[n-3]}_3$, contradicting that $S \in \mathcal A (G)$. If $\alpha^{2} \in \supp(T'_3)$, then $T_3 = \big( \alpha^{r_1 + 1}\tau \bdot \alpha^{r_1 + n}\tau \big)^{[n-1]} \bdot \big( \alpha^{r_2 + 2}\tau \bdot \alpha^{r_2 + n}\tau \big)$ for some $r_1, r_2 \in [0,2n-1]$ (as argued in similar cases). Since $x \equiv 1 + n \pmod{2n}$, we obtain that
\[
  S_1 \, = \, \alpha^{r_2 + n}\tau \bdot \alpha^{r_2 + 2}\tau \bdot \alpha^{r_1 + 1}\tau \bdot \alpha^{r_1 + n}\tau \bdot \alpha^{x}\tau \bdot \tau \quad \und \quad S_2 \, = \, \alpha^{r_1 + n}\tau \bdot \alpha^{r_1 + 1}\tau \bdot \alpha
\]
are both product-one sequences, whence $S = S_1 \bdot S^{[n-2]}_2$, contradicting that $S \in \mathcal A (G)$.

Suppose that $T^{-1}_1 \bdot T_2 \bdot T'_3 = \alpha^{[2n-2]}$ and $x \equiv 2 + n \pmod{2n}$. If $|T_1| \ge 1$, then
\[
  \big( \alpha \bdot \alpha^{-1} \big)^{[|T_1|]} \quad \und \quad \alpha^{[2+\frac{2n-4-2|T_1|}{2}]} \bdot \tau \bdot \alpha^{[\frac{2n-4-2|T_1|}{2}]} \bdot \alpha^{x}\tau
\]
are both product-one sequences, contradicting that $S \in \mathcal A (G)$. Thus $|T_1| = 0$, and we obtain that $T_3 = \big( \alpha^{r + 1}\tau \bdot \alpha^{r + n}\tau \big)^{[n]}$ for some $r \in [0,2n-1]$ (as argued in similar cases). Since $x \equiv 2 + n \pmod{2n}$, we obtain that
\[
  S_1 \, = \, \tau \bdot \alpha^{x}\tau \bdot \big( \alpha^{r + 1}\tau \bdot \alpha^{r + n}\tau \big)^{[2]} \quad \und \quad S_2 \, = \, \alpha^{r + n}\tau \bdot \alpha^{r + 1}\tau \bdot \alpha
\]
are both product-one sequences, whence $S = S_1 \bdot S^{[n-2]}_2$, contradicting that $S \in \mathcal A (G)$.
\end{proof}


\bigskip
\section{Unions of sets of lengths} \label{5}
\bigskip

In this section, we study sets of lengths and their unions in the monoid $\mathcal B (G)$ of product-one sequences over   Dihedral and Dicyclic groups. To do so, we briefly gather the required concepts in the setting of atomic monoids.

Let $H$ be an atomic monoid, this means a  commutative, cancellative semigroup with unit element such that every non-unit element can be written as a finite product of atoms.
If $a = u_1 \cdot \ldots \cdot u_k \in H$, where $k \in \N$ and $u_1, \ldots, u_k$ are atoms of $H$, then $k$ is called the length of the factorization and
\[
  \mathsf L (a) \, = \, \{ k \in \N \t a \, \text{ has a factorization of length } \, k \} \, \subset \, \N
\]
is the {\it set of lengths} of $a$. As usual we set $\mathsf L (a) = \{0\}$ if $a$ is invertible, and then
\[
  \mathcal L (H) \, = \, \{ \mathsf L (a) \t a \in H \}
\]
denotes the {\it system of sets of lengths} of $H$. If $k \in \N$ and  $H$ is not a group, then
\[
  \mathcal U_k (H) \, = \, \bigcup_{k \in L, L \in \mathcal L (H)} L \, \subset \, \N
\]
denotes the {\it union of sets of lengths} containing $k$.
For every $k \in \N$, $\rho_k (H) = \sup \mathcal U_k (H)$ is  the {\it $k$th-elasticity} of $H$, and we denote by $\lambda_k (H) = \inf \mathcal U_k (H)$. Moreover,
\[
  \rho (H) \, = \, \sup \bigg\{ \frac{\rho_k (H)}{k} \t k \in \N \bigg\} \, = \, \underset{k \to \infty}{\lim} \frac{\rho_k (H)}{k}
\]
is the {\it elasticity} of $H$. Unions of sets of lengths have been studied in settings ranging from power monoids to Mori domains and to local quaternion orders (for a sample of recent results we refer to \cite{Ge-Ka-Re15a, F-G-K-T17, Tr18a,Fa-Tr18a,Ba-Sm18}).

Let $G$ be a finite group. The monoid $\mathcal B (G)$ of product-one sequences over $G$ is a finitely generated reduced monoid, and it is a Krull monoid if and only if $G$ is abelian (\cite[Proposition 3.4]{Oh18a}). If $G$ is abelian, then most features of the arithmetic of a general Krull monoid having class group $G$ and prime divisors in all classes can be studied in the monoid $\mathcal B (G)$. For this reason, $\mathcal B (G)$ has received extensive investigations (see \cite{Sc16a} for a survey). If $G$ is non-abelian, then $\mathcal B (G)$ fails to be Krull but it is still a C-monoid (\cite[Theorem 3.2]{Cz-Do-Ge16a}). Thus it shares all arithmetical finiteness properties valid for abstract C-monoids (\cite{Ge-HK06a, Ge-Zh19c}). Investigations aiming at precise results for arithmetical invariants were started in \cite{Oh18a, Oh19a}. We continue them in this section and obtain explicit upper and lower bounds in the case of Dihedral and Dicyclic groups.
As usual, we set
\[
  \mathcal L (G) \, = \, \mathcal L \big( \mathcal B (G) \big), \quad \mathcal U_k (G) \, = \, \mathcal U_k \big( \mathcal B (G) \big), \quad \rho_k (G) \, = \, \rho_k \big( \mathcal B (G) \big), \und \rho (G) \, = \, \rho \big( \mathcal B (G) \big)
\]
for every $k \in N$. It is well-known that $\mathcal U_k (G) = \{ k \}$ for all $k \in \N$ if and only if $|G| \le 2$.
Thus, whenever convenient, we will assume that $|G| \ge 3$. It is already known that the sets $\mathcal U_k (G)$ are intervals (\cite[Theorem 5.5.1]{Oh18a}). Our study of the minima $\lambda_k (G)$ runs along the lines of what was done in the abelian case (\cite[Section 3.1]{Ge09a}). The study of the maxima $\rho_k (G)$ substantially uses the results of Section \ref{4}.

\smallskip
\begin{lemma} \label{5.1}~
Let $G$ be a finite group with $|G| \ge 3$ and let $k \in \N$.
\begin{enumerate}
\item $\rho_k (G) \le \frac{k\mathsf D (G)}{2}$ and $\rho_{2k} (G) = k\mathsf D (G)$. In particular, $\rho (G) = \frac{\mathsf D (G)}{2}$.

\smallskip
\item If $j, l \in \N_0$ such that $l\mathsf D (G) + j \ge 1$, then
      \[
        2l + \frac{2j}{\mathsf D (G)} \,\, \le \,\, \lambda_{l\mathsf D (G)+j} (G) \,\, \le \,\, 2l +j \,.
      \]
      In particular, $\lambda_{l\mathsf D (G)} (G) = 2l$ for every $l \in \N$.
\end{enumerate}
\end{lemma}

\begin{proof}
1. \cite[Proposition 5.6]{Oh18a}.

\smallskip
2. Let $j, l \in \N_0$ such that $l\mathsf D (G) + j \ge 1$. Note that there is some $L \in \mathcal L (G)$ with $k, \lambda_k (G) \in L$, and it follows that
\[
  k \le \max{L} \le \rho (G) \min{L} = \rho (G) \lambda_k (G) \,.
\]
Hence we obtain that
\[
  2l + \frac{2j}{\mathsf D (G)} \,\, = \,\, \rho (G)^{-1} (l\mathsf D (G) + j) \,\, \le \,\, \lambda_{l\mathsf D (G) +j} \,.
\]
Since $2 \le \mathsf D (G)$, it follows by 1. that
\[
  \lambda_{2l+j} (G) \, \le \, 2l + j \, \le \, l \mathsf D (G) + j \, \le \, \rho_{2l} (G) + \rho_j (G) \, \le \, \rho_{2l + j} (G) \,,
\]
whence $l \mathsf D (G) + j \in \mathcal U_{2l+j} (G)$ (by \cite[Theorem 5.5.1]{Oh18a}), equivalently $2l + j \in \mathcal U_{l\mathsf D (G) + j} (G)$. Therefore
\[
  2l + \frac{2j}{\mathsf D (G)} \,\, \le \,\, \lambda_{l\mathsf D (G) +j} \,\, \le \,\, 2l + j \,.
\]
If $j = 0$, then $\lambda_{l\mathsf D (G)} (G) = 2l$.
\end{proof}

\smallskip
\begin{lemma} \label{5.2}~
Let $G$ be a finite group with $|G| \ge 3$. For every $j \in \N_{\ge 2}$, the following statements are equivalent{\rm \,:}
\begin{enumerate}
\item[(a)] There exists some $L \in \mathcal L (G)$ with $\{ 2, j \} \subset L$.

\smallskip
\item[(b)] $j \le \mathsf D (G)$.
\end{enumerate}
\end{lemma}

\begin{proof}
(a) $\Rightarrow$ (b) If $L \in \mathcal L (G)$ with $\{ 2, j \} \subset L$, then Lemma \ref{5.1}.1 implies that $j \le \sup{L} \le \rho_2 (G) = \mathsf D (G)$.

\smallskip
(b) $\Rightarrow$ (a) If $j \le \mathsf D (G)$, then there exists some $U \in \mathcal A (G)$ with $|U| = \ell \ge j$, say $U = g_1 \bdot \ldots \bdot g_{\ell}$ with $g_1g_2 \ldots g_{\ell} = 1_G$. Then $V = g_1 \bdot \ldots \bdot g_{j-1} \bdot (g_j \cdots g_{\ell}) \in \mathcal A (G)$, and $\{ 2, j \} \subset \mathsf L \big(V\bdot V^{-1}\big)$.
\end{proof}

\smallskip
\begin{proposition} \label{5.3}~
Let $G$ be a finite group with $|G| \ge 3$. For every $l \in \N_0$, we have
\begin{displaymath}
\lambda_{l\mathsf D (G) + j} (G)  = \left\{ \begin{array}{ll}
                                            2l     & \textrm{for $j = 0$}, \\
                                            2l + 1 & \textrm{for $j \in [1, \rho_{2l+1} (G) - l\mathsf D (G)]$},\\
                                            2l + 2 & \textrm{for $j \in [\rho_{2l+1} (G) - l\mathsf D (G) + 1, \mathsf D (G) -1]$},
                                            \end{array} \right.
\end{displaymath}
provided that $l\mathsf D (G) + j \ge 1$.
\end{proposition}

\begin{proof}
Let $l \in \N_0$ and $j \in [0, \mathsf D (G) -1]$ such that $l\mathsf D (G) + j \ge 1$. For $j = 0$, the assertion follows from Lemma \ref{5.1}.2. Let $j \in [1,\mathsf D (G) -1]$. Then Lemma \ref{5.1}.2 implies that
\[
  2l + \frac{2j}{\mathsf D (G)} \,\, = \,\, \frac{l\mathsf D (G) +j}{\rho (G)} \,\, \le \,\, \lambda_{l\mathsf D (G) +j} (G) \,\, \le \,\, 2l + j \,.
\]
For the $j = 1$ case, note that $\rho_{2\ell+1} (G) \ge \rho_{2\ell} (G) + 1 = \ell \mathsf D (G) + 1$, so $j = 1$ forces the second of the three cases to hold, and thus we may assume that $j \ge 2$. Then Lemma \ref{5.2} implies that $\{ 2, j \} \subset \mathsf L (U)$ for some $U \in \mathcal B (G)$, whence $\lambda_j (G) = 2$. Thus we have
\[
  \lambda_{l\mathsf D (G) +j} (G) \,\, \le \,\, \lambda_{l\mathsf D (G)} (G) + \lambda_j (G) \,\, = \,\, 2l + 2 \,,
\]
and hence $\lambda_{l\mathsf D (G) +j} (G) \in [2l + 1, 2l + 2]$.

If $j \in [2, \rho_{2l+1} (G) -l\mathsf D (G)]$, then $l \ge 1$, and by \cite[Theorem 5.5.1]{Oh18a}, we obtain that $l\mathsf D (G) +j \in \mathcal U_{2l+1} (G)$, equivalently $2l+1 \in \mathcal U_{l\mathsf D (G) +j} (G)$. Therefore $\lambda_{l\mathsf D (G) +j} (G) \le 2l+1$ and thus $\lambda_{l\mathsf D (G) + j} = 2l + 1$.

If $j > \rho_{2l+1} (G) - l\mathsf D (G)$, then $l\mathsf D (G) + j > \rho_{2l+1} (G)$, and by \cite[Theorem 5.5.1]{Oh18a}, we obtain that $l\mathsf D (G) + j \notin \mathcal U_{2l+1} (G)$, and that $\lambda_{l\mathsf D (G) + j} (G) > 2l + 1$. Therefore $\lambda_{l\mathsf D (G) +j} (G) = 2l + 2$.
\end{proof}

\smallskip
\begin{theorem} \label{5.4}~
Let $G$ be a dihedral group of order $2n$, where $n \in \N_{\ge 3}$ is odd.
Then, for every $k \in \N_{\ge 2}$ and every $l \in \N_0$, we have $\mathcal U_k (G) = [\lambda_k (G), \rho_k (G)]$,
\begin{displaymath}
\rho_k (G) \, = \, kn \quad \und \quad \lambda_{2ln + j} (G) \,\, = \,\, \left\{ \begin{array}{ll}
                                                                      2l + j & \textrm{for $j \in [0,1]$}, \\
                                                                      2l + 2 & \textrm{for $j \ge 2$ and $l = 0$}, \\
                                                                      2l + 1 & \textrm{for $j \in [2, n]$ and $l \ge 1$},\\
                                                                      2l + 2 & \textrm{for $j \in [n+1, 2n-1]$ and $l \ge 1$},
                                                                     \end{array} \right.
\end{displaymath}
provided that $2ln + j \ge 1$.
\end{theorem}

\begin{proof}
We obtain that $\mathcal U_k (G) = [\lambda_k (G), \rho_k (G)]$ by \cite[Theorem 5.5.1]{Oh18a}. We prove the assertion on $\rho_k (G)$, and then the assertion on $\lambda_{2ln + j} (G)$ follows from Proposition \ref{5.3}.

Let $k \in \N$. If $k$ is even, the assertion follows from Lemma \ref{5.1}.1. For odd $k$, it is sufficient to show that $\rho_3 (G) \ge  3n$. Indeed Lemma \ref{5.1}.1 implies that
\[
  3n + 2kn \,\, \le \,\, \rho_3 (G) + \rho_{2k} (G) \,\, \le \,\, \rho_{2k+3} (G) \,\, \le \,\, \frac{(2k+3)2n}{2} \,\, = \,\, 3n + 2kn \,,
\]
and hence the assertion follows.

Since $n \in \N_{\ge 3}$ is odd, it follows by letting $G = \la \alpha, \tau \ra$ that
\[
  U \, = \, (\alpha\tau)^{[n]} \bdot \tau^{[n]} \,, \quad V \, = \, (\alpha^{2}\tau)^{[n]} \bdot (\alpha\tau)^{[n]} \,, \und W \, = \, (\alpha^{2}\tau)^{[n]} \bdot \tau^{[n]}
\]
are the minimal product-one sequences of length $\mathsf D (G)$ (Theorem \ref{4.1}). Thus we obtain that $\{ 3, 3n \} \subset \mathsf L (U\bdot V\bdot W)$, whence $\rho_3 (G) \ge 3n$.
\end{proof}

\smallskip
\begin{theorem} \label{5.5}~
Let $G$ be either a dihedral group $D_{2n}$ of order $2n$ or a dicyclic group $Q_{4m}$ of order $4m$, where $n \in \N_{\ge 4}$ is even and $m \in \N_{\ge 2}$. Then, for every $k \in \N$, we have
\[
  k \mathsf D (G) + 2 \,\, \overset{(a)}{\le} \,\, \rho_{2k+1} (G) \,\, \overset{(b)}{\le} \,\, k \mathsf D (G) + \frac{\mathsf D (G)}{2} -1 \,.
\]
In particular, if $G$ is isomorphic to $D_8$ or to  $Q_8$, then $\rho_{2k+1} (G) = k \mathsf D (G) + 2$ for every $k \in \N$.
\end{theorem}

\begin{proof}
1. Let $n \in \N_{\ge 4}$ be even, and $G = \la \alpha, \tau \t \alpha^{n} = \tau^{2} = 1_G \und \tau\alpha = \alpha^{-1}\tau \ra$.
To show the inequality ($a$), we take three minimal product-one sequences
\[
  U \, = \, \alpha^{[n + \frac{n}{2} - 2]} \bdot \tau \bdot \alpha^{\frac{n}{2}}\tau \,, \quad V \, = \, \big(\alpha^{-1}\big)^{[n+\frac{n}{2}-2]} \bdot \alpha\tau \bdot \alpha^{\frac{n}{2}+1}\tau \,, \und W \, = \, \tau \bdot \alpha^{\frac{n}{2}}\tau \bdot \alpha\tau \bdot \alpha^{\frac{n}{2}+1}\tau
\]
of length $|U| = |V| = \mathsf D (G)$ (Theorem \ref{4.2}) and $|W| = 4$. Then it follows by $\{ 3, \mathsf D (G) + 2 \} \subset \mathsf L (U \bdot V \bdot W)$ that $\mathsf D (G) + 2 \le \rho_3 (G)$, whence we obtain that, for every $k \ge 2$,
\[
  k \mathsf D (G) + 2 \, = \, (k-1) \mathsf D (G) + \big( \mathsf D (G) + 2 \big) \, \le \, \rho_{2k-2} (G) + \rho_3 (G) \, \le \, \rho_{2k+1} (G) \,.
\]

To show the inequality ($b$), we assume to the contrary that $\rho_{2k+1} (G) = \blfloor \frac{(2k+1)\mathsf D (G)}{2} \brfloor$. Then there exist $U_1, \ldots, U_{2k+1} \in \mathcal A (G)$ with $|U_1| \ge \ldots \ge |U_{2k+1}|$ such that $\rho = \rho_{2k+1} (G) \in \mathsf L \big( U_1 \bdot \ldots \bdot U_{2k+1} \big)$. Hence we have that
\[
  U_1 \bdot \ldots \bdot U_{2k+1} \, = \, W_1 \bdot \ldots \bdot W_{\rho} \,,
\]
where $W_1, \ldots, W_{\rho} \in \mathcal A (G)$ with $|W_1| \le \ldots \le |W_{\rho}|$.
Let $H_0 = \la \alpha \ra \setminus \{ 1_G, \alpha^{\frac{n}{2}} \}$. For every $g \in H_0$ and every sequence $S \in \mathcal F (G)$, we define
\[
  \psi_g (S) \, = \, \mathsf v_g (S) - \mathsf v_{g^{-1}} (S) \,.
\]
Then, for every $g \in H_0$, we have $|\psi_g (T)| \le |T|$ and $|\psi_g (W)| = 0$ for sequences $T \in \mathcal F (G)$ and $W \in \mathcal A (G)$ with $|W| = 2$.

\medskip
\noindent
{\bf CASE 1.} \, $|U_1| = \ldots = |U_{2k+1}| = \mathsf D (G)$.
\smallskip

Then we obtain that either $|W_1| = \ldots = |W_{\rho}| = 2$, or else $|W_1| = \ldots = |W_{\rho-1}| = 2$ and $|W_{\rho}| = 3$. Since $2k + 1$ is odd, it follows by Theorem \ref{4.2} that there exists $g_0 \in H_0$ with $\ord(g_0)=n$ such that the absolute value $| \psi_{g_0} (U_1 \bdot \ldots \bdot U_{2k+1}) |$ is $t(\frac{3n}{2} - 2)$ for some $t \in \N$. Since $\psi_{g_0} (W_i) = 0$ for all $i \in [1,\rho-1]$, we obtain that
\[
  4 \, \le \, \Big( \frac{3n}{2} - 2 \Big) \, \le \, | \psi_{g_0} (U_1 \bdot \ldots \bdot U_{2k+1}) | \, = \, | \psi_{g_0} (W_1 \bdot \ldots \bdot W_{\rho}) | \, \le \, | \psi_{g_0} (W_1 \bdot \ldots \bdot W_{\rho - 1}) | + | \psi_{g_0} (W_{\rho}) | \, \le \, 3 \,,
\]
a contradiction.

\medskip
\noindent
{\bf CASE 2.} \, $|U_1| = \ldots = |U_{2k}| = \mathsf D (G)$ and $|U_{2k+1}| = \mathsf D (G) - 1$.
\smallskip

Then we obtain that $|W_1| = \ldots = |W_{\rho}| = 2$ and hence
 \[
  \psi_g (U_1 \bdot \ldots \bdot U_{2k}) + \psi_g (U_{2k+1}) \, = \, \psi_g (U_1 \bdot \ldots \bdot U_{2k+1}) \, = \, \psi_g (W_1 \bdot \ldots \bdot W_{\rho}) \, = \, 0
\]
for every $g \in H_0$. Let $U_{2k+1} = T_1 \bdot T_2$, where $T_1 \in \mathcal F \big( \la \alpha \ra \big)$ and $T_2 \in \mathcal F \big( G \setminus \la \alpha \ra \big)$. If $|T_1| = 0$, then it follows by Proposition \ref{3.2} that $\frac{3n}{2} - 1 = |U_{2k+1}| = |T_2| \le n$, contradicting that $n \ge 4$. If $|T_2| = 0$, then $\mathsf D \big( \la \alpha \ra \big) = n$ ensures that $\frac{3n}{2} - 1 = |U_{2k+1}| = |T_1| \le n$, again a contradiction. Thus $T_1$ and $T_2$ are both non-trivial sequences, and we show that they are product-one sequences to get a contradiction.

First, we prove that $T_1$ is a product-one sequence. Note that $\psi_g (U_{2k+1}) = \psi_g (T_1)$ for all $g \in H_0$. If there exists $g_0 \in H_0$ such that $\psi_{g_0} (T_1) \neq 0$, then $| \psi_{g_0} (U_1 \bdot \ldots \bdot U_{2k}) | = | \psi_{g_0} (T_1) | \ge 1$. Thus Theorem \ref{4.2} ensures that $| \psi_{g_0} (U_1 \bdot \ldots \bdot U_{2k}) | = t(\frac{3n}{2} - 2)$ for some $t \in \N$. Since $|T_2| \ge 2$, it follows that
\[
  \frac{3n}{2} - 1 \, = \, |U_{2k+1}| \, = \, |T_2| + |T_1| \, \ge \, 2 + |\psi_{g_0} (T_1)| \, = \, 2 + t \Big( \frac{3n}{2} - 2 \Big) \, \ge \, \frac{3n}{2} \,,
\]
a contradiction. Thus $\psi_g (U_{2k+1}) = \psi_g (T_1) = 0$ for all $g \in H_0$. Since $\alpha^{\frac{n}{2}} \in \mathsf Z (G)$, we have $\mathsf v_{\alpha^{\frac{n}{2}}} (U) \le 1$ for any $U \in \mathcal A (G)$ with $|U|\ge 3$. Hence Theorem \ref{4.2} ensures that $\alpha^{\frac{n}{2}} \notin \supp(U_i)$ for all $i \in [1,2k]$, and hence $\mathsf v_{\alpha^{\frac{n}{2}}} (U_1 \bdot \ldots \bdot U_{2k+1}) = \mathsf v_{\alpha^{\frac{n}{2}}} (U_{2k+1}) \le 1$. Since $\mathsf v_{\alpha^{\frac{n}{2}}} (W_1 \bdot \ldots \bdot W_{\rho})$ must be even, we obtain $\mathsf v_{\alpha^{\frac{n}{2}}} (U_{2k+1}) = 0$, and therefore $T_1 = \prod^{\bullet}_{i \in [1,|T_1|/2]} (g_i \bdot g^{-1}_i) \in \mathcal B \big( H_0 \big)$.

Next, we show that $T_2$ is a product-one sequence. Let $U_1 \bdot \ldots \bdot U_{2k} = Z_1 \bdot Z_2$, where $Z_1 \in \mathcal F \big( \la \alpha \ra \big)$ and $Z_2 \in \mathcal F \big( G \setminus \la \alpha \ra \big)$. Then Theorem \ref{4.2} implies that
\[
  Z_2 \, = \, V_1 \bdot \ldots \bdot V_{2k} \,,
\]
where for each $i \in [1,2k]$, $V_i = \alpha^{r_i}\tau \bdot \alpha^{\frac{n}{2}+r_i}\tau$ for some $r_i \in [0,n-1]$.
Choose $I \subset [1,2k]$ to be maximal such that $\prod^{\bullet}_{i \in I} V_i$ is a product of minimal product-one sequences of length $2$. Then both $|I|$ and $| [1,2k] \setminus I |$ are even, and thus $Z'_2 = \prod^{\bullet}_{j \in [1,2k] \setminus I} V_j$ is a product-one sequence.

Since $T_1 \bdot Z_1$ is a product of minimal product-one sequences of length $2$, it follows that $T_2 \bdot Z_2$ is also a product of minimal product-one sequences of length $2$. Let $T'_2$ be a subsequence of $T_2$ obtained by deleting all minimal product-one subsequences of length $2$. Then $T'_2 \bdot Z'_2$ is again a product of minimal product-one sequences of length $2$.
Since $T'_2$ and $Z'_2$ are both square-free sequences, we obtain that $T'_2 = Z'_2$ is a product-one sequence, whence $T_2 = \big( T_2 \bdot (T'_2)^{[-1]} \big) \bdot T'_2 \in \mathcal B (G)$.

\smallskip
2. Let $m \ge 2$, and $G = \la \alpha, \tau \t \alpha^{2m} = 1_G, \, \tau^{2} = \alpha^{m}, \und \tau\alpha = \alpha^{-1}\tau \ra$. To show the inequality ($a$), we take three minimal product-one sequences
\[
  U \, = \, \alpha^{[3m-2]} \bdot \tau^{[2]} \,, \quad V \, = \, (\alpha^{-1})^{[3m-2]} \bdot (\alpha\tau)^{[2]} \,, \und W \, = \, (\alpha^{m}\tau \bdot \alpha^{m+1}\tau)^{[2]}
\]
of length $|U| = |V| = \mathsf D (G)$ (Theorem \ref{4.3}) and $|W| = 4$. Then it follows by $\{ 3, \mathsf D (G) + 2 \} \subset \mathsf L (U \bdot V \bdot W)$ that $\mathsf D (G) + 2 \le \rho_3 (G)$, whence we obtain that, for every $k \ge 2$,
\[
  k \mathsf D (G) + 2 \, = \, (k-1) \mathsf D (G) + \big( \mathsf D (G) + 2 \big) \, \le \, \rho_{2k-2} (G) + \rho_3 (G) \, \le \, \rho_{2k+1} (G) \,.
\]

To show the inequality ($b$), we assume to the contrary that $\rho_{2k+1} (G) = \blfloor \frac{(2k+1)\mathsf D (G)}{2} \brfloor$. Then there exist $U_1, \ldots, U_{2k+1} \in \mathcal A (G)$ with $|U_1| \ge \ldots \ge |U_{2k+1}|$ such that $\rho = \rho_{2k+1} (G) \in \mathsf L \big( U_1 \bdot \ldots \bdot U_{2k+1} \big)$. Hence we have that
\[
  U_1 \bdot \ldots \bdot U_{2k+1} \, = \, W_1 \bdot \ldots \bdot W_{\rho} \,,
\]
where $W_1, \ldots, W_{\rho} \in \mathcal A (G)$ with $|W_1| \le \ldots \le |W_{\rho}|$.
Let $H_0 = \la \alpha \ra \setminus \{ 1_G, \alpha^{m} \}$. For every $g \in H_0$ and every sequence $S \in \mathcal F (G)$, we define
\[
  \psi_g (S) \, = \, \mathsf v_g (S) - \mathsf v_{g^{-1}} (S) \,.
\]
Then, for every $g \in H_0$, we have $|\psi_g (T)| \le |T|$ and $|\psi_g (W)| = 0$ for sequences $T \in \mathcal F (G)$ and $W \in \mathcal A (G)$ with $|W| = 2$.

\medskip
\noindent
{\bf CASE 1.} \, $|U_1| = \ldots = |U_{2k+1}| = \mathsf D (G)$.
\smallskip

Then we obtain that either $|W_1| = \ldots = |W_{\rho}| = 2$, or else $|W_1| = \ldots = |W_{\rho-1}| = 2$ and $|W_{\rho}| = 3$. Since $2k+1$ is odd, it follows by Theorem \ref{4.3} that there exists $g_0 \in H_0$ with $\ord(g_0) = 2m$ such that the absolute value $| \psi_{g_0} (U_1 \bdot \ldots \bdot U_{2k+1}) |$ is $t(3m - 2)$ for some $t \in \N$. Since $\psi_{g_0} (W_i) = 0$ for all $i \in [1,\rho-1]$, we obtain that
\[
  4 \, \le \, 3m - 2 \, \le \, | \psi_{g_0} (U_1 \bdot \ldots \bdot U_{2k+1}) | \, = \, | \psi_{g_0} (W_1 \bdot \ldots \bdot W_{\rho}) | \, \le \, | \psi_{g_0} (W_1 \bdot \ldots \bdot W_{\rho - 1}) | + | \psi_{g_0} (W_{\rho}) | \, \le \, 3 \,,
\]
a contradiction.

\medskip
\noindent
{\bf CASE 2.} \, $|U_1| = \ldots = |U_{2k}| = \mathsf D (G)$ and $|U_{2k+1}| = \mathsf D (G) - 1$.
\smallskip

Then we obtain that $|W_1| = \ldots = |W_{\rho}| = 2$, and hence
\[
  \psi_g (U_1 \bdot \ldots \bdot U_{2k}) + \psi_g (U_{2k+1}) \, = \, \psi_g (U_1 \bdot \ldots \bdot U_{2k+1}) \, = \, \psi_g (W_1 \bdot \ldots \bdot W_{\rho}) \, = \, 0
\]
for every $g \in H_0$. Let $U_{2k+1} = T_1 \bdot T_2$, where $T_1 \in \mathcal F \big( \la \alpha \ra \big)$ and $T_2 \in \mathcal F \big( G \setminus \la \alpha \ra \big)$. If $|T_2| = 0$, then $\mathsf D \big( \la \alpha \ra \big) = 2m$ ensures that $3m - 1 = |U_{2k+1}| = |T_1| \le 2m$, a contradiction to $m \ge 2$. Thus $T_2$ is a non-trivial sequence. We show that $T_1$ and $T_2$ are both product-one sequences, and it will be shown that $T_2 \notin \mathcal A (G)$ when $|T_1| = 0$.

First, we prove that $T_1$ is a product-one sequence. Note that $\psi_g (U_{2k+1}) = \psi_g (T_1)$ for all $g \in H_0$. If there exists $g_0 \in H_0$ such that $\psi_{g_0} (T_1) \neq 0$, then $| \psi_{g_0} (U_1 \bdot \ldots \bdot U_{2k}) | = |\psi_{g_0} (T_1)| \ge 1$. Thus Theorem \ref{4.3} ensures that $|\psi_{g_0} (U_1 \bdot \ldots \bdot U_{2k})| = t(3m - 2)$ for some $t \in \N$. Since $|T_2| \ge 2$, it follows that
\[
  3m - 1 \, = \, |U_{2k+1}| \, = \, |T_2| + |T_1| \, \ge \, 2 + |\psi_{g_0} (T_1)| \, = \, 2 + t(3m - 2) \, \ge \, 3m \,,
\]
a contradiction. Thus $\psi_g (U_{2k+1}) = \psi_g (T_1) = 0$ for all $g \in H_0$. Since $\alpha^{m} \in \mathsf Z (G)$, we have $\mathsf v_{\alpha^{m}} (U) \le 1$ for any $U \in \mathcal A (G)$ with $|U| \ge 3$. Hence Theorem \ref{4.3} ensures that $\alpha^{m} \notin \supp(U_i)$ for all $i \in [1,2k]$, and thus $\mathsf v_{\alpha^{m}} (U_1 \bdot \ldots \bdot U_{2k+1}) = \mathsf v_{\alpha^{m}} (U_{2k+1}) \le 1$. Since $\mathsf v_{\alpha^{m}} (W_1 \bdot \ldots \bdot W_{\rho})$ must be even, we obtain $\mathsf v_{\alpha^{m}} (U_{2k+1}) = 0$, and therefore $T_1 = \prod^{\bullet}_{i \in [1,|T_1|/2]} (g_i \bdot g^{-1}_i) \in \mathcal B (H_0)$.

Next, we show that $T_2$ is a product-one sequence, which is not a minimal product-one sequence when $|T_1| = 0$. Let $U_1 \bdot \ldots \bdot U_{2k} = Z_1 \bdot Z_2$, where $Z_1 \in \mathcal F \big( \la \alpha \ra \big)$ and $Z_2 \in \mathcal F \big( G \setminus \la \alpha \ra \big)$. Then Theorem \ref{4.3} implies that
\[
  Z_2 \, = \, V_1 \bdot \ldots \bdot V_{2k} \,,
\]
where for each $i \in [1,2k]$, $V_i = (\alpha^{r_i}\tau)^{[2]}$ for some $r_i \in [0,2m-1]$. Choose $I \subset [1,2k]$ to be maximal such that $\prod^{\bullet}_{i \in I} V_i$ is a product of minimal product-one sequences of length $2$. Then both $|I|$ and $| [1,2k] \setminus I |$ are even, and thus $Z'_2 = \prod^{\bullet}_{j \in [1,2k] \setminus I} V_j$ is a product-one sequence, which is in fact a product of product-one subsequences of length at most $4$.

Since $T_1 \bdot Z_1$ is a product of minimal product-one sequences of length $2$, it follows that $T_2 \bdot Z_2$ is also a product of minimal product-one sequences of length $2$. Let $T'_2$ be a subsequence of $T_2$ obtained by deleting all minimal product-one subsequences of length $2$. Then $T'_2 \bdot Z'_2$ is again a product of minimal product-one sequences of length $2$. Since both $T'_2$ and $Z'_2$ have no product-one subsequences of length $2$ and $\alpha^{m} \in \mathsf Z (G)$, it follows that $1_G \in \pi(Z'_2) = \pi(T'_2)$, whence $T_2 = \big( T_2 \bdot (T'_2)^{[-1]} \big) \bdot T'_2 \in \mathcal B (G)$. 
To conclude the proof, we may assume that $|T_1| = 0$. Then $U_{2k+1} = T_2$, and it follows that either that $T'_2$ is trivial, or that $U_{2k+1} = T'_2$. In the former case, $U_{2k+1}$ is a product of product-one subsequences of length $4$ (as this is the case for $Z'_2$ with the terms of $Z'_2$ and $T'_2$ pairing up), so $U_{2k+1} \in \mathcal A (G)$ forces $3m - 1 = |U_{2k+1}| \le 4$, contradicting that $m \ge 2$. In the latter case, $U_{2k+1}$ is a product of product-one sequences of length $2$ by definition of $T'_2$, whence $U_{2k+1} \in \mathcal A (G)$ forces $3m - 1 = |U_{2k+1}| \le 2$, again a contradiction.
\end{proof}


\providecommand{\bysame}{\leavevmode\hbox to3em{\hrulefill}\thinspace}
\providecommand{\MR}{\relax\ifhmode\unskip\space\fi MR }
\providecommand{\MRhref}[2]{%
  \href{http://www.ams.org/mathscinet-getitem?mr=#1}{#2}
}
\providecommand{\href}[2]{#2}

\end{document}